\documentclass[12pt]{amsart}
\usepackage{latexsym}
\usepackage{amscd, amsfonts,
  eucal, mathrsfs, amsmath, amssymb, amsthm}
\input xy
\xyoption{all}

\usepackage{baskervald}
\usepackage{newcent}
\usepackage{comment}
\usepackage{fullpage}
\usepackage{tikz-cd}
\usepackage[cal=boondoxo, scr=boondoxo]{mathalfa}
\usepackage{hyperref}

\usepackage[capitalise]{cleveref}

\usepackage[utf8]{inputenc}

\input cyracc.def
\DeclareFontFamily{U}{russian}{}
\DeclareFontShape{U}{russian}{m}{n}
        { <5><6> wncyr5
        <7><8><9> wncyr7
        <10><10.95><12><14.4><17.28><20.74><24.88> wncyr10 }{}
\DeclareSymbolFont{Russian}{U}{russian}{m}{n}
\DeclareSymbolFontAlphabet{\mathcyr}{Russian}
\makeatletter
\let\@math@cyr\mathcyr
\renewcommand{\mathcyr}[1]{\@math@cyr{\cyracc #1}}
\makeatother

\DeclareMathSymbol{:}{\mathpunct}{operators}{"3A}

\renewcommand{\phi}{\varphi}

\newcommand{\simto}{\stackrel{\sim}{\to}}

\DeclareMathOperator{\Spec}{Spec}

\DeclareMathOperator{\D}{D}

\renewcommand{\mathbb}{\mathbf}

\newtheorem{lem}[subsection]{Lemma}

\numberwithin{equation}{subsection}

\newtheorem{thm}[subsection]{Theorem}
\newtheorem{assume}[subsection]{Assumption}

\newtheorem{prop}[subsection]{Proposition}

\newtheorem{cor}[subsection]{Corollary}

\theoremstyle{definition}

\newtheorem{example}[subsection]{Example}

\newtheorem{remark}[subsection]{Remark}
\theoremstyle{remark}

\newtheorem{pg}[subsection]{}

\crefname{thm}{Theorem}{Theorems}
\crefname{pg}{Paragraph}{Paragraphs}
\crefname{prop}{Proposition}{Propositions}

\newcommand{\mls}[1]{\mathscr{#1}}

\newcommand{\Sp}{\text{\rm Spec}}

\newcommand{\lotimes}{\otimes ^{\mathbf{L}}}

\newcommand{\qcoh}{\mathrm{qcoh}}

\author{Max Lieblich and Martin Olsson}

\title{Derived equivalences over base  schemes and support of complexes}

\begin{document}
\begin{abstract}
Let $X$ and $Y$ be smooth projective varieties over a field $k$ admitting morphisms $f:X\rightarrow T$ and $g:Y\rightarrow T$ to a third variety $T$.  We formulate conditions on a derived equivalence $\Phi :D(X)\rightarrow D(Y)$ ensuring  that $\Phi$ is induced by a complex $P\in D(X\times _TY)$, defining derived equivalences between the fibers of $f$ and $g$.   We apply our results to the canonical fibration and albanese fibration.
\end{abstract}
\maketitle

\setcounter{tocdepth}{1}
\tableofcontents

\section{Introduction}
\label{sec:introduction}

  Let $X$ and $Y$ be derived equivalent varieties over a field $k$ with equivalence given by a complex $P\in  D(X\times Y)$.  Suppose further given morphisms $f:X\rightarrow T$ and $g:Y\rightarrow T$ to a scheme $T$.  We would like to understand conditions under which we can conclude that the derived equivalence restricts to equivalences on the fibers of $f$ and $g$, at least over some open subset of $T$.  More concretely, we would like to understand conditions on the equivalence that will ensure that $P$ is in the image of $D(X\times _TY)$.  This problem also presents itself in the work of Toda \cite{Toda}.  Our interest in this problem arose, in part, from establishing cases where derived equivalent varieties are birational using canonical morphisms to other varieties.

In general, it appears to be a difficult question to decide when a complex set-theoretically supported on a
closed subscheme $X_0\subset X$ is the pushforward of a complex from $X_0$.
Analogous questions about affine morphisms $X'\to X$ have an interesting
history. As shown in \cite[Theorem 8.1]{rizzardo}, even for base change by field extensions,
it is unusual for a complex with an $X'$ structure to be pushed forward from
$X'$.  
For kernels of derived equivalences, however, we have additional tools at our disposal.
First, on the level of $\infty $-categories a very satisfactory solution was given by Ben-Zvi, Francis, and Nadler \cite[4.7]{BZFN}.  Second, to get results on the level of derived categories, rather than $\infty $-categories, we use variations of the gluing results of Beilinson, Bernstein, and Deligne developed in \cite{BO}.

Our main technical results are not restricted to derived equivalences and we state them in more general form.  Let $S$ be a scheme (in practice this will often be the spectrum of a field) and let
$$
f:X\rightarrow T, \ \ g:Y\rightarrow T
$$
be separated morphisms of $S$-schemes with $T/S$ flat.    Let 
$$
\delta _0, \delta _1:X\times _SY\rightarrow X\times _ST\times _SY
$$
be the morphisms given on scheme-valued points
$$
\delta _0(x, y) = (x, f(x), y), \ \ \delta _1(x, y) = (x, g(y), y).
$$

We consider a pair $(P, \varphi )$ where $P\in D_{\mathrm{qcoh}}(X\times _ST)$ is an object of the derived category of complexes with quasi-coherent cohomology sheaves on $X\times _ST$ and 
$$
\varphi :\delta _{0*}P\rightarrow \delta _{1*}P
$$
is an isomorphism in $D_{\mathrm{qcoh}}(X\times _ST\times _SY)$ such that the pushforward (all functors are derived)
$$
\mathrm{pr}_{13*}\varphi :P \simeq \mathrm{pr}_{13*}\delta _{0*}P\rightarrow \mathrm{pr}_{13*}\delta _{1*}P\simeq P
$$
is the identity morphism.  Note that if
$$
\epsilon :X\times _TY\rightarrow X\times _SY
$$
is the natural inclusion then $\delta _0\circ \epsilon = \delta _1\circ \epsilon $ and therefore for a complex $P_0\in D_{\mathrm{qcoh}}(X\times _TY)$ the pushforward $\epsilon _*P_0$ admits a natural such isomorphism $\varphi $ over $X\times _ST\times _SY$.

\begin{thm}\label{T:1.1}
Assume that either $f$ or $g$ is flat and that $(P, \varphi )$ is a pair as above.  Assume further that $P$ is a perfect complex and that the derived pushforward $\mathrm{pr}_{1*}\mls RHom (P, P)$ lies in $D^{\geq 0}(X)$.  Then there exists a unique pair $(P_0, \lambda )$ consisting of a complex $P_0\in D_{\mathrm{qcoh}}^b(X\times _TY)$ and $\lambda :\epsilon _*P_0\simeq P$ identifying $\varphi $ with the canonical isomorphism between the pushforwards of $P_0$.
\end{thm}

We discuss two main applications:

\begin{pg}[Canonical fibration]
Let $X$ and $Y$ be smooth projective varieties over a field $k$ related by a derived equivalence $\Phi :D(X)\rightarrow D(Y)$ given by a complex $P\in D(X\times Y)$.  Let $R_X:=\oplus _{n\geq 0}\Gamma (X, K_X^{\otimes n})$ (resp. $R_Y:=\oplus _{n\geq 0}\Gamma (Y, K_Y^{\otimes n})$) denote the canonical ring of $X$ (resp. $Y$) so we have rational maps
$$
c_X:\xymatrix{X\ar@{-->}[r]& \mathrm{Proj}(R_X)}, \ \ c_Y:\xymatrix{Y\ar@{-->}[r]& \mathrm{Proj}(R_Y)}.
$$
It is well-known (this is stated explicitly in \cite[4.4]{Toda} and attributed to \cite{orlov} in \cite[6.1]{huybrechtsFM}) that $\Phi $ induces a canonical isomorphism 
\begin{equation}\label{E:caniso}
\tilde \tau :R_X\rightarrow R_Y,
\end{equation}
and therefore also an isomorphism of schemes
$$
\tau :\mathrm{Proj}(R_X)\simeq \mathrm{Proj}(R_Y).
$$
Let $U_X\subset X$ (resp. $U_Y\subset Y$) be the complement of the base locus of $\{K_X^{\otimes n}\}_{n\geq 0}$ (resp. $\{K_Y^{\otimes n}\}_{n\geq 0}$), so $c_X$ (resp. $c_Y$) is a morphism on $U_X$ (resp. $U_Y$).  Also, for an open subset $A\subset \mathrm{Proj}(R_Y)$ let $U_{X, A}$ (resp. $U_{Y, A}$) denote the preimage of $A$ under $\tau \circ c_X$ (resp. $c_Y$).  
\end{pg}

\begin{thm}\label{T:1.3} There exists a dense open subset $A\subset \mathrm{Proj}(R_Y)$ such that the restriction of $P$ to $U_{X, A}\times Y$ is the image of an object $$
P_0\in D(U_{X, A}\times _{\tau \circ c_X, A, c_Y}U_{Y, A})
$$
whose support is proper over both $U_{X, A}$ and $U_{Y, A}$.
\end{thm}

\begin{remark} (i) The above generalizes earlier work of Toda \cite[1.1]{Toda}.

(ii) Over fields of characteristic $0$ the canonical ring is known to be finitely generated.  Over fields of positive characteristic, however, this is not known so a little more care is needed in the arguments presented below.

(iii) We also prove a version of the above theorem replacing $U_X$ with the maximal open subset $W_X$ over which the morphism $c_X$ extends.
\end{remark}

\begin{pg}[Albanese fibration]\label{P:Albanese}
Let $X$ be a smooth projective variety over a field $k$, let $\mathrm{Pic}^0(X)$ (resp. $\mathrm{Aut}^0(X)$) denote the connected component of the identity of the Picard scheme of $X$ (resp. the automorphism group scheme of $X$), and set
$$
\mathbf{R}_X^0:= \mathrm{Pic}^0(X)\times \mathrm{Aut}^0(X).
$$
As we recall in section \ref{S:sec4}, if 
$$
\Phi :D(X)\rightarrow D(Y)
$$
is a derived equivalence given by a complex $P\in D(X\times Y)$, then $\Phi $ induces an isomorphism
$$
\Phi _{\mathbf{R}^0}:\mathbf{R}_X^0\simeq \mathbf{R}_Y^0.
$$

Let
$$
c_X:X\rightarrow \mathbf{T}_X^0, \ \ c_Y:Y\rightarrow \mathbf{T}_Y^0
$$
be the Albanese torsors of $X$ and $Y$ (see section \ref{S:sec4} for more discussion).  For an open subset $U\subset \mathbf{T}_X^0$ let $X_U$ denote $c_X^{-1}(U)$, and similarly for $Y$.
\end{pg}

\begin{thm}[Theorem \ref{T:4.9}]\label{T:1.6} Assume that $\mathrm{Pic}^0_X$ and $\mathrm{Pic}^0_Y$ are reduced and that $\Phi _{\mathbf{R}^0}$ sends $\mathrm{Pic}^0_X$ to $\mathrm{Pic}^0_Y$ and therefore defines an isomorphism
$$
\Phi _{\mathrm{Pic}^0}:\mathrm{Pic}^0_X\rightarrow \mathrm{Pic}^0_Y.
$$
Then $\Phi $ induces an isomorphism of schemes 
$$
\Phi _{\mathbf{T}^0}:\mathbf{T}_X^0\rightarrow \mathbf{T}_Y^0
$$
compatible with the actions of the Picard schemes, and there exists a dense open subset $A\subset \mathbf{T}_Y^0$ such that 
$$
P|_{X_{\Phi _{\mathbf{T}^0}^{-1}(A)}\times Y_A}
$$
is in the image of
$$
D(X_{\Phi _{\mathbf{T}^0}^{-1}(A)}\times _AY_A).
$$
\end{thm}

\begin{remark}
(a) The assumption that the groups schemes $\mathrm{Pic}^0_X$ and $\mathrm{Pic}^0_Y$ are reduced holds for example if $k$ has characteristic $0$.

(b) The assumption that $\Phi _{\mathbf{R}^0}$ preserves the Picard schemes frequently holds.  For example, if the automorphism group schemes are affine.  Varieties for which the automorphism group scheme has a nontrivial abelian variety as quotient can be classified;  see \cite[2.4]{popaschnell}.
\end{remark}

\begin{pg} The body of the article is divided into three sections.  Section \ref{sec:derived supp} is devoted to the proof of \ref{T:1.1}.  The key ingredient is \ref{T:adjunction}, which reduces the proof to a problem of gluing in the derived category of a cosimplicial scheme. Theorem \ref{T:1.1} is obtained from this and a variant of the BBD gluing lemma for cosimplicial schemes.  Section \ref{S:sec3} is concerned with the canonical fibration.  In this section we prove, in particular, theorem \ref{T:1.3}.  In addition to the results of \ref{sec:derived supp} we use in a key way the Beilinson resolution of the diagonal of projective space.  Finally in section \ref{S:sec4} we apply a similar analysis to the albanese fibration proving \ref{T:1.6}.
\end{pg}
 
 \subsection{Acknowledgments}
The authors have benefited from conversations with many people.  We thank, in 
particular, 
Dan Bragg, Denis-Charles Cisinski, Katrina Honigs, Johan de Jong,  Christian Schnell, Paolo Stellari, Sofia Tirabassi, and Gabriele Vezzosi.

During the work on this paper, Lieblich was partially supported by NSF grant 
DMS-1600813 and a Simons Foundation Fellowship, and Olsson was partially
supported by NSF grants DMS-1601940 and DMS-1902251.

\section{Support of complexes and relativization of equivalences}\label{sec:derived supp}

\subsection{Complexes on fiber products}

For the convenience of the reader we review the result \cite[4.7]{BZFN} using the more classical language of derived categories.

\begin{pg}\label{P:setup} We fix a base scheme $S$, a flat $S$-scheme $T/S$,  and consider two separated
morphisms of $S$-schemes 
$$
f:X\rightarrow T, \ \ g:Y\rightarrow T.
$$
Let $\gamma _f:X\rightarrow X\times _ST$ (resp. $\delta _g:Y\rightarrow T\times _SY$) be the maps given on scheme-valued points by
$$
\gamma _f(x) = (x, f(x)), \ \ \delta _g(y) = (g(y), y).
$$
We then get a cosimplicial scheme 
$$
E(f, g)^\bullet 
$$
by the following variant of the bar construction.  Set
$$
E(f, g)^n = X\times _ST^{\times n}\times _SY,
$$
where $T^{\times n}$ denotes the $n$-fold fiber product of $T$ with itself over $S$ and we make the convention that $T^0 = S$ so $E(f, g)^0 = X\times _SY$.
Define maps
$$
\delta _i:E(f, g)^{n-1}\rightarrow E(f, g)^n, \ \ i=0, \dots, n
$$
by 
$$
\delta _0 = \gamma _f\times \mathrm{id}_{T^{\times (n-1)}\times Y}, \ \ \delta _n = \mathrm{id}_{X\times T^{\times (n-1)}}\times \delta _g, 
$$
and, for $i=1, \dots, n-1$, 
$$
\delta _i = \mathrm{id}_{X\times T^{\times (i-1)}}\times \Delta _T\times \mathrm{id}_{T^{\times (n-i-1)}\times Y}.
$$
Define maps 
$$
\sigma _i:E(f, g)^{n+1}\rightarrow E(f, g)^n, \ \ i=0, \dots, n,
$$
by the formula
$$
\sigma _i(x, t_1, \dots, t_{n+1}, y) = (x, t_1, \dots, \hat {t}_{i+1}, \dots, t_{n+1}, y)
$$
on scheme-valued points.
\end{pg}

\begin{lem} The maps $\delta _i$ and $\sigma _i$ satisfy the cosimplicial identities \cite[\href{https://stacks.math.columbia.edu/tag/016K}{Tag 016K}]{stacks-project} and therefore define a cosimplicial scheme $E(f, g)^\bullet $.
\end{lem}
\begin{proof} 
This is immediate from the definitions.
\end{proof}

\begin{pg} We have an augmentation
$$
\epsilon :X\times _TY\rightarrow E(f, g)^\bullet,
$$
that is, a map from $X\times_T Y$ to $E(g,f)^0$ compatible with each map in the cosimplicial structure.
Writing 
$$
\epsilon _n:X\times _TY\rightarrow E(f, g)^n
$$
for the induced map in degree $n$, the map $\epsilon _n$ is given on scheme valued points by
$$
(x, y)\mapsto (x, f(x), \dots, f(x), y).
$$
Note that since $f(x) = g(y)$ we could also write this formula using $g(y)$. One can also think of the augmentation as a morphism from the constant cosimplicial scheme on $X\times_T Y$.
\end{pg}

\begin{pg}

Let $\mathrm{Mod}(E(f, g)^\bullet )$ denote the category of systems $(\{\mls F_n\}, \varphi _\delta )$, where $\mls F_n$ is a  sheaf of $\mls O_{E(f, g)^n}$-modules on the \'etale site of  $E(f, g)^n$ and for every morphism $\delta :E(f, g)^n\rightarrow E(f, g)^m$ given by the cosimplicial structure we have a morphism
$$
\varphi _\delta :F_m\rightarrow \delta _*F_n,
$$
and these morphisms are compatible with compositions.  The category $\mathrm{Mod}(E(f, h)^\bullet )$ is abelian, with kernels and cokernels defined level-wise (in particular, restriction to any particular $E(f,g)^n$ is an exact functor).  We let 
$$
D(E(f, g)^\bullet )
$$
denote the associated derived category and 
$$
D^{(-)}(E(f, g)^\bullet )\subset D(E(f, g)^\bullet )
$$
the subcategory of complexes whose restriction to each $E(f, g)^n$ is bounded above.  

For an object $\mls P\in D(E(f, g)^\bullet )$ we let $\mls P_n$ denote its restriction to $E(f, g)^n$.  For a morphism $\delta :E(f, g)^n\rightarrow E(f, g)^m$ we have a map 
$$
\varphi _\delta :\mls P_m\rightarrow R\delta _*\mls P_n
$$
in the derived category of $E(f, g)^m$.
We write
$$
D_\qcoh  (E(f, g)^\bullet )\subset D(E(f, g)^\bullet )
$$
for the subcategory of complexes for which the sheaves $\mls H^i(\mls P_m)$ are quasi-coherent and the maps
$\varphi _\delta $ are all isomorphisms, and 
$$
D^-_\qcoh (E(f, g)^\bullet ):= D^{(-)}(E(f, g)^\bullet )\cap D_\qcoh (E(f, g)^\bullet ).
$$
Pushforward along the augmentation defines a functor
\begin{equation}\label{E:pushmap}
\epsilon _*:D^-_\qcoh (X\times _TY)\rightarrow D^-_\qcoh (E(f, g)^\bullet ).
\end{equation}
\end{pg}

\begin{thm}\label{T:adjunction} If either $f$ or $g$ is flat then the functor \eqref{E:pushmap} is an equivalence.
\end{thm}

\begin{remark} It seems likely that one can formulate a version of this result
without the flatness assumption replacing $X\times _TY$ by a suitable derived
fiber product.
\end{remark}

The proof will be in several steps.

\begin{pg}
The functor
$$
\epsilon _*:\mathrm{Mod}(X\times _TY)\rightarrow \mathrm{Mod}(E(f, g)^\bullet )
$$
has a left adjoint given by pullback along the canonical inclusion $X\times_T Y\to X\times _SY$.  Deriving this left adjoint we get a functor
$$
\mathbf{L}\epsilon ^*:D^-_\qcoh (E(f, g)^\bullet )\rightarrow D^-_\qcoh (X\times _TY).
$$
More concretely, the functor $\mathbf{L}\epsilon ^*$ is calculated by 
$$
\mathbf{L}\epsilon ^*\mls P = \mathrm{hocolim}_{\Delta }\mathbf{L}\epsilon _n^*\mls P_n.
$$
We show that the functors
$$
\mathbf{L}\epsilon ^*\epsilon _*:D^-_\qcoh (X\times _TY)\rightarrow D^-_\qcoh (X\times _TY)
$$
and
$$
\epsilon _*\mathbf{L}\epsilon ^*:D^-_\qcoh  (E(f, g)^\bullet )\rightarrow D^-_\qcoh  (E(f, g)^\bullet )
$$
are isomorphic to the respective identity functors by the adjunction maps. Note that the functors and natural transformations are all compatible with restriction to open subschemes. In particular, we can check whether the adjunction maps are isomorphisms locally on $X\times _SY$.
\end{pg}

\begin{pg}
We reduce to the case when $X$, $Y$, $S$, and $T$ are all affine as follows.

The maps $\delta _0, \delta _1:X\times _SY\rightarrow X\times _ST\times _SY$ are closed immersions, since $f$ and $g$ are separated.  For $\mls P\in D_{\mathrm{qcoh}}^-(E(f, g)^\bullet )$ we have
\begin{equation}\label{E:1.8.1}
\delta _0(\mathrm{Supp}(\mls P_0)) = \mathrm{Supp}(\mls P_1) = \delta _1(\mathrm{Supp}(\mls P_0)).
\end{equation}
Now observe that if $U\subset T$ is an open subset then 
$$
\delta _1^{-1}(\delta _0(f^{-1}(U)\times _SY)) = f^{-1}(U)\times _Sg^{-1}(U).
$$
Combining this with \eqref{E:1.8.1} we find that if $T = \cup _iT_i$ is an open covering of $T$ and if 
$$
f_i:X_i\rightarrow T_i, \ \ g_i:Y_i\rightarrow T_i
$$
are the restrictions of $f$ and $g$ then $\mls P$ is supported on 
$$
\cup _iE(f_i, g_i)^\bullet \subset E(f, g)^\bullet ,
$$
and  it suffices to verify that our adjunction maps are isomorphisms for $D(f_i, g_i)^\bullet $.  We may therefore assume that $T$ is affine, say $T=\Spec R$, and that the map $T\rightarrow S$ factors through an affine open subset $\Sp (k)\subset S$ for a ring $k$.  Replacing $S$ by $\Sp (k)$ we are reduced to the case when $S$ and $T$ are affine.

Having made this reduction, we can then cover $X$ and $Y$ by affines and verify that the adjunction maps are isomorphisms over corresponding open subsets.

We may therefore assume that $X = \Sp (A)$ and $Y = \Sp (B)$ for $R$-algebras $A$ and $B$.
\end{pg}

\begin{pg}
In this case $E(f, g)^\bullet $ is given by the simplicial ring $\mls A_\bullet $ (with tensor products taken over $k$)
$$
\xymatrix{
\cdots A\otimes R\otimes R\otimes B\ar@<1ex>[r]\ar[r]\ar@<-1ex>[r]&A\otimes R\otimes B\ar@<1ex>[r]\ar@<-1ex>[r]\ar@<.5ex>[l]\ar@<-.5ex>[l]& A\otimes B.\ar[l]}
$$
\end{pg}

\begin{lem} The augmentation $\mls A_\bullet \rightarrow A\otimes _RB$ induces a quasi-isomorphism on the associated normalized complexes.
\end{lem}
\begin{proof}
Let $\mls A'_\bullet \rightarrow R$ be the simplicial ring with augmentation obtained from the above construction taking $A = B = R$.  

Since  the maps
$$
\delta _i^*:A\otimes R^{\otimes n}\otimes B\rightarrow A\otimes R^{\otimes (n-1)}\otimes B
$$
are $A\otimes B$-linear,  the normalized complex of $\mls A_\bullet $ is isomorphic to the complex obtained from the normalized complex of $\mls A'_\bullet $ tensored over $R\otimes R$ with $A\otimes B$.  Note also that $\mls A'_\bullet $ is term-wise flat over $R\otimes R$.  It follows that if we show that $\mls A'_\bullet \rightarrow R$ induces a quasi-isomorphism on associated normalized complexes, then the map
$$
\mls A_\bullet \simeq \mls A'_{\bullet }\otimes _{R\otimes R}(A\otimes B)\rightarrow R\otimes ^{\mathbf{L}}_{R\otimes R}(A\otimes B)\simeq A\otimes _RB
$$
also induces a quasi-isomorphism (using the flatness of one of $f$ or $g$).  We are therefore reduced to the case when $A = B = R$.

In this case it follows from direct calculation that the maps
$$
h_n:R^{\otimes (n+2)}\rightarrow R^{\otimes (n+3)}, \ \ a_0\otimes \cdots \otimes a_{n+1}\mapsto a_0\otimes \cdots \otimes a_{n+1}\otimes 1
$$
define a homotopy between the identity and $0$.
\end{proof}

\begin{pg} From this and \cite[I, 3.3.4.6]{Illusie} we see that if $P\in D^-(A\otimes _RB)$ then the adjunction map
$$
\mathbf{L}\epsilon ^*\epsilon _*P = (A\otimes _RB)\otimes _{\mls A_\bullet }^{\mathbf{L}}P\rightarrow P
$$
is an equivalence.

To verify that the adjunction
\begin{equation}\label{E:1.10.1}
P\rightarrow \epsilon _*\mathbf{L}\epsilon ^*P
\end{equation}
is an equivalence for $P\in D_\qcoh ^-(E(f, g )^\bullet )$, note that since all the transition maps in $E(f, g)^\bullet $ are affine, and therefore have exact pushforwards, and $\epsilon ^*$ is right exact we get by descending induction that it suffices to consider the case when $P$ is level-wise a module concentrated in degree $0$.  In this case our assumptions imply  that $P$ is, in fact, of the form $\epsilon _*P_0$ for a $A\otimes _RB$-module $P_0$ and we have
$$
\epsilon _*\mathbf{L}\epsilon ^*\epsilon _*P_0\simeq \epsilon _*P_0
$$
by the case already considered. This isomorphism identifies \eqref{E:1.10.1} with the identity map (in the derived category), and therefore \eqref{E:1.10.1} is an isomorphism. This completes the proof of \ref{T:adjunction}. \qed
\end{pg}

\begin{example}  
Let $X = \Sp (R)$ be an affine scheme over a field $k$ and let $f\in R$ be an element defining an effective Cartier divisor $Z\hookrightarrow X$.  Let 
$$
F:X\rightarrow \mathbf{A}^1_k
$$
be the morphism defined by $f$ so that $F^{-1}(0) = Z$.  We can then apply our setup with $T = \mathbf{A}^1_k$, $Y = \Sp (k)$, and $g$ the zero section $\Sp (k)\hookrightarrow \mathbf{A}^1_k$.  The two maps
$$
\delta _0, \delta _1:X = X\times Y\rightarrow X\times T\times Y  = X\times \mathbf{A}^1_k
$$
are then given by 
$$
\delta _0 = (\mathrm{id}_X, F),  \ \ \delta _1 = (\mathrm{id}_X, 0);
$$
that is, $\delta _0 =\gamma _F$ is the graph of $F$ and $\delta _1 = \gamma _0$ is the graph of the zero map.
\end{example}

\subsection{Proof of \ref{T:1.1}}
\cref{T:adjunction} reframes the problem of descending a complex $P\in D(X\times _S Y)$ to $D(X\times _TY)$ to one of extending $P$ to an object of $D(E(f, g)^\bullet )$, the derived category of the cosimplicial scheme $E(f, g)^\bullet $.  This is, fundamentally, a problem of gluing objects of the derived category (though not in the classical setting of a covering in a site but instead in the setting of gluing objects in a cosimplicial topos) and we apply the results of \cite{BO}.

\begin{pg}
Consider again the setup of \ref{P:setup}. Let $P\in D^b_{\mathrm{qcoh}}(X\times _S Y)$ be a complex equipped with an isomorphism
$$
\varphi :\delta _{0*}P\simto \delta _{1*}P
$$
on $X\times _ST\times _SY$ such that the diagram
\begin{equation}\label{E:2.1.1}
\xymatrix{
t_{2*}P\ar[r]^-{\simeq }\ar[d]^-{\simeq }&  \delta _{0*}\delta _{0*}P\ar[r]^-{\delta _{0*}\varphi }& \delta _{0*}\delta _{1*}P=\delta _{2*}\delta _{0*}P\ar[r]^-{\delta _{2*}\varphi }& \delta _{2*}\delta _{1*}P \ar[r]^-{\simeq }& t_{0*}P\ar[d]^-{\simeq }\\
\delta _{1*}\delta _{0*}P\ar[rrrr]^-{\delta _{1*}\varphi }&&&& \delta _{1*}\delta _{1*}P}
\end{equation}
commutes, 
where for $0\leq i\leq 2 $ we write $t_i:[0]\rightarrow [2]$ for the unique map with image $i$.  We show that under the assumptions of \ref{T:1.1} there exists a unique pair $(P_0, \lambda )$, where 
$P_0\in D^b_{\mathrm{qcoh}}(X\times _TY)$ is a complex and $\lambda :\epsilon _{0*}P_0\simeq P$ is an isomorphism identifying $\varphi $ with the canonical isomorphism between the pushforwards of $P_0$.
\end{pg}
\begin{pg}
  Define $P_n\in D^b_{\mathrm{qcoh}}(E(f, g)^n)$ to be the pushforward of $P$ along the morphism $X\times _SY\rightarrow E(f, g)^n$ given by the map $[0]\rightarrow [n]$ sending $0$ to $0$.

For $m\geq 0$ let $\gamma _i:[0]\rightarrow [m]$ ($0\leq i\leq m$) be the morphism in $\Delta $ sending $0$ to $i$.
For a morphism $[n]\rightarrow [m]$ in $\Delta $ let 
$$
\alpha _\delta :[1]\rightarrow [m]
$$
be the map sending $0$ to $0$ and $1$ to $\delta (0)$.  So we have
$$
\gamma _0 = \alpha _\delta \circ \delta _0, \ \ \gamma _{\delta (0)} = \alpha _\delta \circ \delta _1.
$$
We then get an isomorphism
$$
\xymatrix{
\varphi _\delta :P_m = Rf_{\alpha _\delta *}\delta _{0*}P\ar[r]^-{\varphi }&Rf_{\alpha _\delta *}\delta _{1*}P\simeq Rf_{\delta *}P_n.}
$$
The maps $\varphi _\delta $ are compatible with composition.  For maps
$$
\xymatrix{
[n]\ar[r]^-\delta & [m]\ar[r]^-\epsilon & [k]}
$$
set
$$
g:[2]\rightarrow [k], \ \ 0\mapsto 0, 1\mapsto \epsilon (0), 2\mapsto \epsilon \delta (0).
$$
Then applying $Rf_{g*}$ to the diagram \eqref{E:2.1.1} gives that
$$
Rf_{\epsilon *}\varphi _\delta \circ \varphi _\epsilon  = \varphi _{\epsilon \delta }.
$$

\end{pg}

\begin{lem} Let $\delta :[n]\rightarrow [m]$ and $\epsilon :[t]\rightarrow [m]$ be  morphisms in $\Delta $, with associated morphisms $f_\delta :E(f, g)^n\rightarrow E(f,g )^m$ and $f_\epsilon :E(E, g)^t\rightarrow E(f, g)^m$.   Then
$$
\mathrm{Ext}^i_{E(f, g)^m}(Rf_{\epsilon *}P_t, Rf_{\delta *}P_n) = 0, \ \ \text{for $i<0$.}
$$
\end{lem}
\begin{proof}
We have $Rf_{\delta *}P_n\simeq P_m\simeq Rf_{\epsilon *}P_t$, so it suffices to show that
$$
\mathrm{Ext}^s_{E(f, g)^m}(\gamma _{0*}P, \gamma _{0*}P) = 0
$$
for $s<0$.

Let
$$
t:X\rightarrow X\times _ST^m
$$
be the map given on scheme-valued points by
$$
x\mapsto (x, f(x), \dots, f(x)).
$$
Then $\gamma _0$ is obtained by taking the product of $t$ with $Y$.  By adjunction, and using the fact that $P$ is perfect,  we have
$$
\mathrm{Ext}^s_{E(f, g)^m}(\gamma _{0*}P, \gamma _{0*}P)\simeq H^s(X\times Y, \mls RHom(\mathbf{L}\gamma _{0}^*\gamma _{0*}\mls O_{X\times Y}, \mls O_{X\times Y})\otimes \mls RHom (P, P)).
$$
Let $\mls R$ denote the complex
$$
\mls RHom (\mathbf{L}t^*t_*\mls O_{X}, \mls O_X).
$$
Then we have
$$
\mathrm{pr}_1^*\mls R\simeq \mls RHom(\mathbf{L}\gamma _{0}^*\gamma _{0*}\mls O_{X\times Y}, \mls O_{X\times Y}).
$$
We conclude that
\begin{equation}\label{E:pushforward}
\mls R\otimes ^{\mathbf{L}}R\mathrm{pr}_{1*}\mls RHom (P, P)\simeq R\mathrm{pr}_{1*}(\mls RHom(\mathbf{L}\gamma _{0}^*\gamma _{0*}\mls O_{X\times Y}, \mls O_{X\times Y})\otimes \mls RHom (P, P)).
\end{equation}
Since $\mls R$ is locally represented by a complex of projective $\mls O_X$-modules concentrated in degrees $\geq 0$ and $R\mathrm{pr}_{1*}\mls RHom (P, P)$ is concentrated in degrees $\geq 0$ the complex \eqref{E:pushforward} is in $D^{\geq 0}(X)$.  In particular, its cohomology is zero in negative degrees.
\end{proof}

\cref{T:1.1} now follows from the lemma and the BBD gluing lemma for a $D$-topos \cite[1.4]{BO}. \qed

\section{Canonical fibration}\label{S:sec3}

\begin{pg}
For a smooth projective variety $X$ over a field $k$ let 
$$
R_X:= \oplus _{n\geq 0}\Gamma (X, K_X^{\otimes n})
$$
be its canonical ring.  Over fields of characteristic $0$ this ring is known to be finitely generated \cite{BCHM}.

For an integer $n\geq 1$ for which $H_{X, n}:= \Gamma (X, K_X^{\otimes n})$ is nonzero we get a rational map
$$
\pi _n:\xymatrix{X\ar@{-->}[r]& \mathbf{P}H_{X, n}}.
$$
We then get open subsets
$$
U_{X, n}\subset W_{X, n}\subset X,
$$
where $U_{X, n}$ is the maximal open subset over which $H_{X, n}$ generates $K_X^{\otimes n}$ and $W_{X, n}$ is the maximal open subset over which $\pi _n$ is a morphism.  Since $X$ is normal the complement of $W_{X, n}$ in $X$ has codimension $\geq 2$ and the invertible sheaf $\pi _n^*\mls O_{\mathbf{P}H_{X, n}}(1)$ extends uniquely to an invertible subsheaf
$$
\pi _n^*\mls O_{\mathbf{P}H_{X, n}}(1)\hookrightarrow K_X^{\otimes n}
$$
over all of $X$ for which there is a map
$$
H_{X, n}\rightarrow \Gamma (X, \pi _n^*\mls O_{\mathbf{P}H_{X, n}}(1))
$$
whose image generates $\pi _n^*\mls O_{\mathbf{P}H_{X, n}}(1)$ over $W_{X, n}$.

For an open subset $A\subset \mathbf{P}H_{X, n}$ we write 
$U_{X, n, A}$ (resp. $W_{X, n, A}$)
for the preimage of $A$ in $U_{X, n}$ (resp. $W_{X, n}$).
\end{pg}

\begin{pg}
If $X$ and $Y$ are two smooth projective varieties over $k$ related by a derived equivalence 
$$
\Phi :D(X)\rightarrow D(Y)
$$
then $\Phi $ induces an isomorphism 
\begin{equation}\label{E:3.2.1}
H_{X, n}\simeq H_{Y, n}
\end{equation}
for all $n$; in particular, $\Phi $ induces an isomorphism of canonical rings $R_X\simeq R_Y$.  This is due to Bondal and Orlov \cite{bondalorlov01}.

Let us recall the argument.  For an integer $n\in \mathbf{Z}$ define a functor
$$
S_n:D(X)\rightarrow D(X), \ \ \mls F\mapsto \mls F\otimes K_X^{\otimes n}.
$$
Define $\mls S_{K_X}$ to be the category whose objects are the functors $S_n$ and for which the morphisms $S_m\rightarrow S_n$ are given by elements of $H^0(X, K_X^{\otimes (n-m)})$.  So $\mls S_{K_X}$ is a subcategory of the category $\mathrm{End}(D(X))$ of endofunctors of $D(X)$.
\end{pg}

\begin{lem}\label{P:3.0.8} Let $X$ and $Y$ be smooth projective varieties over 
a field $k$ and let $\Phi :\D(X)\rightarrow \D(Y)$ be an equivalence of 
triangulated categories.  Then the induced functor
\begin{equation}\label{E:3.8.1}
\mathrm{End}(\D(X))\rightarrow \mathrm{End}(\D(Y)), \ \ F\mapsto \Phi \circ 
F\circ 
\Phi ^{-1}
\end{equation}
sends $\mls S_{K _X}$ to $\mls S_{K _Y}$.
\end{lem}
\begin{proof}
The fact that conjugation by $\Phi$ matches up the objects of the categories $\mls S_{K_X}$ and $\mls S_{K_Y}$ is due to Bondal and Orlov \cite{bondalorlov01}.
Let $P\in D(X\times Y)$ be a complex defining $\Phi $.
For an integer $n$ and $S_{X, n}\in \mathrm{End}(D(X))$ (resp. $S_{Y, n}\in \mathrm{End}(D(Y))$) given by tensoring with $K_X^{\otimes n}$ (resp. $K_Y^{\otimes n}$) we have
$\Phi \circ S_{X, n}$
given by $P\otimes p_X^*K_X^{\otimes n}$ and $S_{Y, n}\circ \Phi $ given by $P\otimes p_Y^*K_Y^{\otimes n}$, where $p_X$ and $p_Y$ are the projections.  The result therefore follows from the standard fact \cite[5.22]{huybrechtsFM} that
$$
P\otimes p_X^*K_X^{\otimes n}\simeq P\otimes p_Y^*K_Y^{\otimes n}.
$$

Since these subcategories are not full, however, a bit more is required to get the compatibility on morphisms.  
Following \cite{toen}, let $L_{perf}(X)$ (resp. $L_{perf}$) denote the 
dg-category of perfect  complexes of quasi-coherent sheaves on $X$.  The kernel 
$P$ then defines an equivalence
$$
\widetilde \Phi :L_{perf}(X)\rightarrow L_{perf}(Y).
$$
Let $S_X:\D(X)\rightarrow \D(X)$ be the Serre functor of $X$.  By the uniqueness 
part of Orlov's theorem, as well as Toën's representability result in 
\cite[8.15]{toen} the functor has a lift
$$
\widetilde S_X:L_{perf}(X)\rightarrow L_{perf}(X)
$$
which is unique up to equivalence of dg functors (in the sense of \cite{toen}). 
 
In fact, $\widetilde S_X$ is given by $\Delta _{X*}\omega _X\in 
L_{perf}(X\times 
X)$.  For integers $n$ and $m$ it therefore makes sense to consider the subspace
$$
\mathrm{Hom}'_{\mathrm{End}(\D(X))}(S_X^n, S_X^m)\subset 
\mathrm{Hom}_{\mathrm{End}(\D(X))}(S_X^n, S_X^m)
$$
of morphisms of functors $S_X^n\rightarrow S_X^m$ which admit liftings to 
morphisms of dg functors $\widetilde S_X^n\rightarrow \widetilde S_X^m$.  By 
\cite[8.9]{toen} the set $\mathrm{Hom}'_{\mathrm{End}(\D(X))}(S_X^n, S_X^m)$ 
consists precisely of those morphisms induced by sections of $K_X^{\otimes 
(m-n)}$.

Now for a lift $\widetilde S_X$ the functor
$$
\widetilde \Phi \circ \widetilde S_X\circ \widetilde \Phi 
^{-1}:L_{perf}(Y)\rightarrow L_{perf}(Y)
$$
is a dg lift of the Serre functor $S_Y$ of $Y$.  From this it follows that 
\eqref{E:3.8.1} sends
$$\mathrm{Hom}'_{\mathrm{End}(\D(X))}(S_X^n, S_X^m)$$ to 
$\mathrm{Hom}'_{\mathrm{End}(\D(Y))}(S_Y^n, S_Y^m)$ which implies the lemma.
\end{proof}

\begin{thm}\label{T:3.4}
Let $X$ and $Y$ be smooth projective varieties over $k$ and let $\Phi :D(X)\rightarrow D(Y)$ be a derived equivalence.  Let $n\geq 1$ be an integer such that $H_{X, n}$ (and therefore also $H_{Y, n}$) is nonzero.

(a)  The support of $P|_{U_{X, n}\times Y}$ (resp. $P|_{W_{X, n}\times Y}$) is contained in $U_{X, n}\times U_{Y, n}$ (resp. $W_{X, n}\times W_{Y, n}$).

(b) There exists a dense open subset $A\subset \mathbf{P}H_{X, n}$ such that $P|_{W_{X, n, A}\times Y}$ is in the image of
$$
D(W_{X, n, A}\times _AW_{Y, n, A})\rightarrow D(W_{X, n, A}\times Y).
$$
\end{thm}

The proof occupies the remainder of this section.

\begin{remark}
Note that the open subset $U_X\subset X$ considered in \ref{T:1.3} is the union over all $n$ of the $U_{X, n}$.  Since $U_X$ is quasi-compact, we in fact have $U_X = U_{X, n}$ for $n>>0$ and therefore \ref{T:1.3} follows from \ref{T:3.4}.
\end{remark}

\subsection{The complex $\mls C_{X, n}$.}

\begin{pg}\label{P:4.16b}
We can define a complex $\mls C_{X, n}$ on $X$ with a map of 
complexes 
$$
\epsilon _{X, n}:\mls C_{X, n}\rightarrow \pi _n^*\mls O_{\mathbf{P}H_{X, n}}(1)
$$
which restricts to a quasi-isomorphism over $W_{X, n}$.  Recall that we write $\pi _n^*\mls O_{\mathbf{P}H_{X, n}}(1)$ for the line bundle on $X$ obtained by pullback under the rational map $\pi _n$.  This complex $\mls C_{X, n}$ will be used to understand the set $W_{X, n}$.

The complex $\mls C_{X, n}$ is the Koszul complex associated to the map 
$H_{X, n}\otimes _k\mls O_X\rightarrow K_X^{\otimes n}$ (note that this map factors through $\pi _n^*\mls O_{\mathbf{P}H_{X, n}}(1)$).  Precisely, we have 
$$
\mls C^i_{X, n}:= (\wedge ^{-i+1}H_{X, n})\otimes _kK_X^{\otimes (in)}
$$
for $i\leq 0$ and $\mls C^i_{X, n} = 0$ for $i>0$.  The differential 
$$
d_i:\mls C^{i}_{X, n}\rightarrow \mls C^{i+1}_{X, n}
$$
is given by the usual formula in local coordinates
$$
d_i((h_1\wedge \cdots \wedge h_i)\otimes \ell ):= \sum 
_{j=1}^i(-1)^{j+1}(h_1\wedge \cdots \hat h_j\cdots \wedge h_i)\otimes (\rho 
(h_i)\otimes \ell ),
$$
where $\rho :H_{X, n}\otimes _k\mls O_X\rightarrow K_X$ is the natural map.  The map 
$\epsilon _{X, n}$ is defined to be the map induced by the natural map 
$H_{X, n}\otimes _k\mls O_X\rightarrow \pi _n^*\mls O_{\mathbf{P}H_{X, n}}(1).$  By 
standard properties of the Koszul complex the restriction of $\epsilon _{X, n}$ 
to $W_{X, n}$ is a quasi-isomorphism. 

If 
\begin{equation}\label{E:4.17.1}
\Sigma _n\subset \pi _n^*\mls O_{\mathbf{P}H_{X, n}}(1)
\end{equation}
is the image of $H_{X, n}$ then a point $z\in X$ lies in $W_{X, n}$ if and only if $\Sigma _{n, z}$ is generated by a single element.  Indeed if this is the case then $\Sigma _n$ is a line bundle in a neighborhood of $z$ and the inclusion \eqref{E:4.17.1} restrict to this open subset to an isomorphism, since it is an inclusion of line bundles which is an isomorphism away from a codimension $2$ subset.
\end{pg}

\begin{pg}
For integers $n<m$ and a section $\alpha \in H^0(X, K_{X}^{\otimes (m-n)})$, 
multiplication by $\alpha $ induces a map  $H_{X, n}\rightarrow H_{X, m}$.  This map 
induces a map
$$
\gamma _\alpha :\mls C_{X, n}\rightarrow \mls C_{X, m}
$$
of complexes.  Define 
$$
\mls Kos_X\subset \mathrm{End}(\D(X))
$$
to be the full subcategory whose objects are the functors $\Phi _{X, n}$ given by tensor 
product with the complexes $\mls C_{X, n}$ and whose morphisms are given by 
sections of $K_X^{\otimes r}$ as above.
\end{pg}

\begin{prop}\label{P:3.34} Let $X$ and $Y$ be smooth projective varieties over 
$k$ and let $$\Phi :\D(X)\rightarrow \D(Y)$$ be a derived equivalence given by a 
kernel $P\in \D(X\times Y)$.  Then the essential image of $\mls Kos_X$ under the 
functor
$$
\mathrm{End}(\D(X))\rightarrow \mathrm{End}(\D(Y)), \ \ F\mapsto \Phi \circ 
F\circ 
\Phi ^{-1}
$$
is equal to $\mls Kos_Y$.
\end{prop}
\begin{proof}
Let $p:X\times Y\rightarrow X$ and $q:X\times Y\rightarrow Y$ be the 
projections.  The proof of \cref{P:3.0.8} implies that there exist isomorphisms
$$
\sigma _n:p^*K_X^{\otimes n}\lotimes P\simeq P\lotimes q^*K_Y^{\otimes n}
$$
in $\D(X\times Y)$, such that for any $n<m$ and $\alpha \in H^0(X, 
K_{X}^{\otimes 
(m-n)})$ the diagram
$$
\xymatrix{
p^*K_X^{\otimes n}\lotimes P\ar[r]^-\alpha \ar[d]^-{\sigma _n}& p^*K_X^{\otimes 
m}\lotimes P\ar[d]^-{\sigma _m}\\
p^*K_Y^{\otimes n}\lotimes P\ar[r]^-{\tilde \tau (\alpha )}& p^*K_Y^{\otimes 
m}\lotimes 
P}
$$
commutes, where $\tilde \tau (\alpha )$ is the image of $\alpha $ under the isomorphism \eqref{E:3.2.1}.  To prove the proposition it suffices to extend these isomorphisms to an isomorphism of complexes
\begin{equation}\label{E:3.8.0}
\lambda :p^*\mls C_{X, n}\lotimes P\simeq q^*\mls C_{Y, n}\lotimes P.
\end{equation}
Indeed if $T_{X, n}\in \mls Kos_X$ (resp. $T_{Y, n}\in \mls Kos_Y$) represents the functor defined by $\mls C_{X, n}$ (resp. $\mls C_{Y, n}$) then such an isomorphism defines an isomorphism
$$
\Phi \circ T_{X, n}\simeq T_{Y, n}\circ \Phi .
$$

For an integer $s$ define $\mls C_{X, n}^{\leq s}$ to be the complex which in 
degrees $i\leq s$ is the same as $\mls C_{X, n}$ but which has zero terms in 
degree $>s$.  There is then a distinguished triangle for each $s$
\begin{equation}\label{E:3.34.1}
\xymatrix{
\mls C_{X, n}^s[-s]\ar[r]& \mls C_{X, n}^{\leq s}\ar[r]& \mls C_{X, n}^{\leq 
s-1}\ar[r]& \mls C_{X, n}^s[-s+1].
}
\end{equation}


To prove the proposition we construct for each $s$ an isomorphism in $\D(X\times 
Y)$
$$
\lambda ^{\leq s}:p^*\mls C_{X, n}^{\leq s}\lotimes P\simeq q^*\mls C_{Y, 
n}^{\leq s}\lotimes P,
$$
such that the diagram
\begin{equation}\label{E:3.34.2}
\xymatrix{
\mls C^s_{X, n}[-s]\lotimes P\ar[rrr]^-{\wedge ^{-s+1}(\tilde \tau )\otimes \sigma 
_{ns}}\ar[d]& & &\mls C^s_{Y, n}[-s]\lotimes P\ar[d]\\
\mls C^{\leq s}_{X, n}\lotimes P\ar[rrr]^-{\lambda ^{\leq s}}& & &\mls C^{\leq 
s}_{Y, 
n}\lotimes P}
\end{equation}
commutes.

\begin{lem}\label{L:3.35}  Let $s$, $i$, and $j$ be integers with $j>s$.
\begin{enumerate}
    \item [(i)] We have 
    $$
    \mathrm{Hom}_{\D(X\times Y)}(p^*\mls C_{X, n}^{\leq s}\lotimes P, q^*\mls 
C_{Y, n}^i\lotimes P[-j]) = 0.
    $$
    \item [(ii)] The restriction map
    $$
    \mathrm{Hom}_{\D(X\times Y)}(p^*\mls C_{X, n}^{\leq s}\lotimes P, q^*\mls 
C_{Y, n}^i\lotimes P[-s])\rightarrow \mathrm{Hom}_{\D(X\times Y)}(p^*\mls C_{X, 
n}^{s }\lotimes P[-s], q^*\mls C_{Y, n}^i\lotimes P[-s])
    $$
    is injective.
\end{enumerate}
\end{lem}
\begin{proof}
By considering the distinguished triangles \eqref{E:3.34.1} the proof of (i) is 
reduced to showing that for all integers $s$, $i$, and $j>s$ we have
$$
\mathrm{Hom}_{\D(X\times Y)}(p^*\mls C_{X, n}^s\lotimes P[-s], q^*\mls C_{Y, 
n}^{i}\lotimes P[-j]) = 0.
$$
This follows from noting that elements of this group correspond to morphisms of 
functors
$$
\Phi \circ \Phi ^{\mls C_{X, n}^s[-s]}\rightarrow \Phi ^{\mls C_{Y, 
n}^i[-j]}\circ \Phi 
$$
which can be lifted to the dg-categories of complexes of coherent sheaves.  
Using the isomorphism
$$
\Phi ^{\mls C_{Y, n}^i[-j]}\circ \Phi \simeq \Phi \circ \Phi ^{\mls C_{X, 
n}^i[-j]}
$$
and applying $\Phi ^{-1}$ we see that we have to show that there are no nonzero 
morphisms of functors
$$
\Phi ^{\mls C_{X, n}^s[-s]}\rightarrow \Phi ^{\mls C_{X, n}^i[-j]}
$$
which can be lifted to the dg-category.  Here for a complex $K\in \D(X)$ we 
write 
$\Phi ^K$ for the endofunctor given by tensoring with $K$, and similarly for 
complexes on $Y$.  Equivalently, we need to show that  there are no nonzero 
morphisms in $\D(X\times X)$ 
$$
\Delta _{X*}\mls C_{X, n}^s[-s]\rightarrow \Delta _{X*}\mls C_{X, n}^i[-j],
$$
which follows from the fact that $j>s$.

Statement (ii) follows from (i) and consideration of the triangles 
\eqref{E:3.34.1}.
\end{proof}

We now construct $\lambda ^{\leq s}$ inductively.  For $s$ sufficiently 
negative 
we have $\mls C_{X, n}^{\leq s} = 0$ so there is nothing to show.  So we assume 
that $\lambda ^{\leq s}$ has been defined and construct $\lambda ^{\leq 
(s+1)}$. 
 For this consider the diagram of distinguished triangles
$$
\xymatrix{
\mls C^{s+1}_{X, n}[-(s+1)]P\ar[d]^-{\tau \otimes \sigma }\ar[r]& \mls C_{X, 
n}^{\leq (s+1)}\lotimes P\ar@{-->}[d]\ar[r]& \mls C_{X, n}^{\leq s}\lotimes 
P\ar[d]^-{\lambda _n}\ar[r]& \mls C^{s+1}_{X, n}[-s]P\ar[d]^-{\tau \otimes 
\sigma }\\
\mls C^{s+1}_{Y, n}[-(s+1)]P\ar[r]& \mls C_{Y, n}^{\leq (s+1)}\lotimes P\ar[r]& 
\mls C_{Y, n}^{\leq s}\lotimes P\ar[r]& \mls C^{s+1}_{X, n}[-s]P,}
$$
where the right-most inner square commutes by \cref{L:3.35} (ii).  Now define 
$\lambda ^{\leq (s+1)}$ to be a morphism as indicated by the dotted arrow.  
Note 
that in fact such a morphism is unique by \cref{L:3.35} (i).  This completes 
the 
proof of \cref{P:3.34}.
\end{proof}

\subsection{Proof of \ref{T:3.4} (a)}

If $x\in U_{X, n}$ then the  skyscraper sheaf $\kappa (x)\in \D(X)$ has the property that there exists an element $\alpha \in H_{X, n}$ such that 
$$
\alpha :\kappa (x)\rightarrow \kappa (x)\lotimes K_{X}^{\otimes n}
$$
is an isomorphism.  It follows that $P_x$ has the property that there exists an element $\alpha '\in H_{Y, n}$ for which the map
$$
\alpha ':P_x\rightarrow P_x\lotimes K_Y^{\otimes n}
$$
is an isomorphism. The statement for $P|_{U_{X, n}\times Y}$ follows from this and the following \ref{L:3.12}.

To get the statement for $W_{X, n}$ note that if $x\in W_{X, n}$ is a point then from the equation \eqref{E:3.8.0} we find that
$$
P_x\simeq P_x\lotimes \mls C_{Y, n}.
$$
We get the statement for $P_{W_{x, n}\times Y}$ from this and the following \ref{L:3.13}.

\begin{lem}\label{L:3.12} Let $Q\in \D(Y)$ be a complex such that there exists an element $\alpha \in H_{Y, n}$ for which the induced map
$$
\alpha :Q\rightarrow Q\lotimes K_Y^{\otimes n}
$$
is an isomorphism.  Then the support of $Q$ is contained in $U_{Y, n}$.
\end{lem}
\begin{proof}
Indeed the assumptions imply that for a point $z\in Y$ in the support of $Q$ the fiber $\alpha (z)\in K_Y^{\otimes n}(z)$ is nonzero, and therefore $z\in U_{Y, n}$.
\end{proof}

\begin{lem}\label{L:3.13} Let $Q\in \D(Y)$ be a complex such that 
$$
 Q\otimes ^{\mathbf{L}} \mls C_{Y, n}\simeq Q.
$$
Then the support of $Q$ is contained in  $W_{Y, n}$.
\end{lem}
\begin{proof}
Let $z\in \mathrm{Supp}(Q)$ be a point in the support.  Let $t$ be the largest 
integer for which $\mls H^t(Q)_z\neq 0$.  Since $\mls C_{Y, n}\in \D^{\leq 
0}(Y)$ 
we then have
$$
\mls H^t(Q\otimes ^{\mathbf{L}}\mls C_{Y, n})_z\simeq \mls H^t(Q)_z\otimes 
_{\mls O_{Y, z}}\mls H^0(\mls C_{Y, n}).
$$
We therefore find that
$$
\mls H^t(Q)_z\otimes _{\mls O_{Y, z}}\mls H^0(\mls C_{Y, n})\simeq \mls 
H^t(Q)_z.
$$
Since we assume that $\mls H^t(Q)_z$ is nonzero, this implies, by Nakayama's
lemma, that $\mls H^0(\mls C_{Y, n})_z$ is generated by a single element.  It
follows that the  subsheaf $$\Sigma _n\subset \pi _n^*\mls
O_{\mathbf{P}H_{Y, n}}(1)$$ generated by the image of $H_{Y, n}$ is locally free of rank
$1$ at $z$, which implies that $z\in W_{Y, n}$.
\end{proof}

\subsection{Set-theoretic support}

In order to prove \ref{T:3.4} (b) we will first need a set-theoretic statement.

\begin{lem}\label{L:3.16} Let $\mls F$ be a coherent sheaf on $W_{Y, n}$,  and let  $f:Z\rightarrow W_{Y, n}$ be a morphism, with $Z$ 
proper, such that  $\mls F\otimes f^*\mls C_{Y, n}\simeq \mls F$.  Then  $f(\mathrm{Supp}(\mls F))\subset W_{Y, n}$ is contained 
in a finite union of fibers of $\pi _n$. 
\end{lem}
\begin{proof}
It suffices to prove the lemma after making a base change to an algebraic 
closure of $k$.  Replacing $Z$ be an alteration if necessary we may assume that 
$Z$ is smooth and proper over $k$ and that $\mls F$ is supported on all of $Z$. 
 
Note that over $W_{Y, n}$ we have $\mls C_{Y, n}\simeq \mls O_{Y}(n)$  so $f^*\mls C_{Y, n}\simeq f^*\pi ^*\mls O_{\mathbf{P}H_{Y, n}}(1)$.

Let $r$ be the generic rank of $\mls F$.  Then taking determinants we find that 
$$
\mathrm{det}(\mls F)\simeq \mathrm{det}(\mls F)\otimes f^*\pi ^*\mls O_{\mathbf{P}H_{Y, n}}(r).
$$
Therefore $f^*\pi ^*\mls O_{\mathbf{P}H_{Y, n}}(1)$) is a torsion line bundle on $Z$, which implies that the 
image 
of $Z$ in $\mathbf{P}H_{Y, n}$ is a zero-dimensional subscheme.
\end{proof}

\begin{pg} For a point $x\in W_{X, n}$ the skyscraper sheaf $\kappa (x)$ has the property that 
$$
\kappa (x)\lotimes \mls C_{X, n}\simeq \kappa (x).
$$
It follows that we also have
$$
P_x\lotimes \mls C_{Y, n}\simeq P_x
$$
in $D(Y)$ (Note: we already showed that the support of these complexes lies in $W_{Y, n}$).  By the lemma we conclude that the image of the support of $P_x$ in $\mathbf{P}H_{Y, n}$ lies in a finite number of fibers of $\pi _n$.  And since $\mathrm{End}_{D(Y)}(P_x) = k$ the support is, in fact, connected.  We have shown:
\end{pg}

\begin{cor}\label{C:3.17} The set-theoretic support of $P|_{W_{X, n}\times Y}$ is contained in 
$$
W_{X, n}\times _{\mathbf{P}H_{Y, n}}W_{Y, n},
$$
where the map $W_{X, n}\rightarrow \mathbf{P}H_{Y, n}$ is the composition of $\pi _n:W_{X, n}\rightarrow \mathbf{P}H_{X, n}$ and the isomorphism \eqref{E:3.2.1}.
\end{cor}

\begin{pg}
Though not used in what follows, we also observe that $P$ induces derived equivalences of open varieties as follows.  Note that since             
$$
P_{U, n}:= P|_{U_{X, n}\times U_{Y,n}} \ (\mathrm{resp.} \ P_{W, n}:= P|_{W_{X, n}\times W_{Y, n}})
$$
 has proper support over both 
$U_{X, n}$ and $U_{Y, n}$ (resp. $W_{X, n}$ and $W_{Y, n}$) the complex $P$ induces functors
$$
\Phi _{U, n}:\D(U_{X, n})\rightarrow \D(U_{Y, n}), \ \ \Phi _{W, n}:\D(W_{X, n})\rightarrow \D(W_{Y, n}).
$$
\end{pg}

\begin{prop}\label{P:3.18} The functors $\Phi _{U, n}$ and $\Phi _{W, n}$ are equivalences of 
triangulated categories.
\end{prop}
\begin{proof}
That $\Phi _{U, n}$ is an equivalence can be seen as follows.
  Let $P^\vee \in \D(Y\times X)$ be the complex defining $\Phi 
^{-1}:\D(Y)\rightarrow \D(X)$, and let $P^\vee _U$ be the restriction of $P^\vee 
$ 
to $U_{Y, n}\times U_{X, n}$, which defines
  $$
  \Phi _{U, n}^\vee :\D(U_{Y, n})\rightarrow \D(U_{X, n})
  $$
  We claim that $\Phi _{U, n}\circ \Phi _{U, n}^\vee \simeq \mathrm{id}_{\D(U_{Y, n})}$ and 
$\Phi _{U, n}^\vee \circ \Phi _{U, n} \simeq \mathrm{id}_{\D(U_{X, n})}$.
  
  To see this observe that the restriction of $P$ to $U_{X, n}\times Y$ is equal to 
the pushforward of $P_{U, n}$ by \ref{T:3.4} (a), and similarly for $P^\vee $.  Since 
the diagram
  $$
  \xymatrix{
  X\times Y\times X\ar[d]^-{\mathrm{pr}_{13}}& U_{X, n}\times Y\times 
U_{X, n}\ar[d]^-{\mathrm{pr}_{13}}\ar@{_{(}->}[l]\\
  X\times X& U_{X, n}\times U_{X, n}\ar@{_{(}->}[l]}
  $$
  is cartesian we conclude that the pushforward of $p_{12}^*P_U\otimes 
p_{23}^*P_U^\vee $ along the map
  $$
  p_{13}:U_{X, n}\times U_{Y, n}\times U_{X, n}\rightarrow U_{X, n}\times U_{X, n}
  $$
  is isomorphic to $\Delta _{U_{X, n}*}\mls O_{U_{X, n}}.$  It follows that $\Phi _{U, n}^\vee 
\circ \Phi _{U, n} \simeq \mathrm{id}_{\D(U_{X, n})}$.   The isomorphism $$\Phi _{U, n}\circ 
\Phi 
_{U, n}^\vee \simeq \mathrm{id}_{\D(U_{Y, n})}$$ is shown similarly.

The proof that $\Phi _{W, n}$ is an equivalence follows verbatim from the preceding argument replacing ``$U$'' by ``$W$'' everywhere.
\end{proof}

\subsection{Proof of \ref{T:3.4} (b)}
\begin{pg}
First recall Beilinson's resolution of the diagonal on a projective space $\mathbf{P}(V)$ \cite{Beilinsonresolution}.  This resolution takes the form (let $d$ denote the dimension of $\mathbf{P}(V)$)
$$
0\rightarrow p_1^*\mls O_{\mathbf{P}(V)}(-d)\otimes p_2^*\Omega _{\mathbf{P}(V)} ^d(d)\rightarrow \cdots \rightarrow p_1^*\mls O_{\mathbf{P}(V)}(-1)\otimes p_2^*\Omega ^1_{\mathbf{P}(V)}(1)\rightarrow \mls O_{\mathbf{P}(V)\times \mathbf{P}(V)}\rightarrow \mls O_\Delta \rightarrow 0.
$$
The transition maps are obtained as follow.  We have 
\begin{align*}
& \ \mathrm{Hom}(p_1^*\mls O_{\mathbf{P}(V)}(-i)\otimes p_2^*\Omega ^{i}_{\mathbf{P}(V)}(i), p_1^*\mls O_{\mathbf{P}(V)}(-i+1)\otimes p_2^*\Omega ^{i-1}_{\mathbf{P}(V)}(i-1))\\ & \simeq \mathrm{Hom}_{\mathbf{P}(V)}(\mls O_{\mathbf{P}(V)}(-i), \mls O_{\mathbf{P}(V)}(-i+1))\otimes \mathrm{Hom}_{\mathbf{P}(V)}(\Omega ^i_{\mathbf{P}(V)}(i), \Omega ^{i-1}_{\mathbf{P}(V)}(i-1))\\
& \simeq V\otimes \mathrm{Hom}_{\mathbf{P}(V)}(\Omega ^i_{\mathbf{P}(V)}(i), \Omega ^{i-1}_{\mathbf{P}(V)}(i-1)).
\end{align*}
From the twisted Euler sequence
$$
0\rightarrow \mls O_{\mathbf{P}(V)}(-1)\rightarrow V^\vee \otimes _k\mls O\rightarrow T_{\mathbf{P}(V)}(-1)\rightarrow 0
$$
we obtain an isomorphism
$$
V^\vee \simeq H^0(\mathbf{P}(V), T_{\mathbf{P}(V)}(-1)).
$$
Together with the natural map 
\begin{equation}\label{E:4.30.15}
V^\vee \simeq H^0(\mathbf{P}(V), T_{\mathbf{P}(V)}(-1))\rightarrow \mathrm{Hom}_{\mathbf{P}(V)}(\Omega ^i_{\mathbf{P}(V)}(i), \Omega ^{i-1}_{\mathbf{P}(V)}(i-1))
\end{equation}
we then get a map
$$
V\otimes V^\vee \rightarrow \mathrm{Hom}(p_1^*\mls O_{\mathbf{P}(V)}(-i)\otimes p_2^*\Omega ^{i}_{\mathbf{P}(V)}(i), p_1^*\mls O_{\mathbf{P}(V)}(-i+1)\otimes p_2^*\Omega ^{i-1}_{\mathbf{P}(V)}(i-1)).
$$
The image of the identity class in $V\otimes V^\vee $ defines under this map the differential in the Beilinson resolution.
\end{pg}

\begin{pg}
Returning to the proof of \ref{T:3.4} (b), let $\mathbf{P}\subset \mathbf{P}H_{X, n}$ be the closure of the image of $W_{X, n}$, viewed as a scheme with the reduced-induced structure, 
and let
$$
f:W_{X, n}\rightarrow \mathbf{P}, \ \ g:W_{Y, n}\rightarrow \mathbf{P}
$$
be the natural maps.
\end{pg}

\begin{pg}
Consider first the case when $k$ is infinite.

In this case, for a suitable subspace $V\subset H_{X, n}$ the induced rational map
$$
\xymatrix{
\mathbf{P}\ar@{-->}[r]& \mathbf{P}(V)}
$$
is everywhere defined, finite, and generically \'etale (see for example \cite[2.11]{alterations}).

Let
$$
f':W_{X, n}\rightarrow \mathbf{P}(V), \ \ g':W_{Y}\rightarrow \mathbf{P}(V)
$$
be the induced maps, and let $E(f', g')^\bullet $ be the associated cosimplicial scheme.
We write
$$
\delta _i':W_{X,  n}\times W_{Y, n}\rightarrow E(f', g') = W_{X, n}\times \mathbf{P}(V)\times W_{Y, n}
$$
for the structure maps in this cosimplicial scheme (and similarly for other maps occurring in the cosimplicial structure). 

Note that with the notation of \cref{P:4.16b} we have
$$
f^{\prime *}\mls O_{\mathbf{P}(V)}(1) = \Sigma _{X, n}, \ \ g^{\prime *}\mls O_{\mathbf{P}(V)}(1)= \Sigma _{Y, n}.
$$
Note that these are line bundles on $W_{X, n}$ and $W_{Y, n}$.

For an open set $B\subset \mathbf{P}(V)$ we can also consider the restrictions
$$
f'_B:W_{X, n, B}\rightarrow B, \ \ g'_B:W_{Y, n, B}\rightarrow B,
$$
and the associated cosimplicial scheme
$$
E(f_B', g_B')\hookrightarrow E(f', g').
$$
\end{pg}

\begin{pg}
We have a cartesian diagram
$$
\xymatrix{
W_{X, n}\ar[d]\ar[r]& W_{X, n}\times \mathbf{P}(V)\ar[d]\\
\mathbf{P}(V)\ar[r]^-{\Delta }& \mathbf{P}(V)\times \mathbf{P}(V).}
$$
Pulling back the Beilinson resolution of the diagonal of $\mathbf{P}(V)$ we obtain a complex on $W_{X, n}\times \mathbf{P}(V)$ of the form
\begin{equation}\label{E:4.2.1}
p_1^*\Sigma _{X, n}^{\otimes (-d)}\otimes p_2^*\Omega _{\mathbf{P}(V)} ^d(d)\rightarrow \cdots \rightarrow p_1^*\Sigma _{X, n}^{\otimes (-1)}\otimes p_2^*\Omega ^1_{\mathbf{P}(V)}(1)\rightarrow \mls O_{X\times \mathbf{P}(V)}.
\end{equation}
Over the locus in $W_{X, n}$ where the map $f'$ is flat this is a resolution of $\mls O_{\gamma _f}$, where $\gamma _f:W_{X, n}\rightarrow W_{X, n}\times \mathbf{P}(V)$ is the graph of $f$.  Pulling the complex \eqref{E:4.2.1} back  to $W_{X, n}\times \mathbf{P}(V)\times W_{Y, n}$ along the first two projections we get a complex on $E(f', g')^1$, which over the preimage of the flat locus of $f'$ is a resolution of $\delta _{0*}\mls O_{W_{X, n}\times W_{Y, n}}$.

\end{pg}
\begin{pg}
For $s\leq 0$ set 
\begin{equation}\label{E:4.3.1}
P_X^s:= p_{13}^*(p_1^*\Sigma _{X, n}^{\otimes {s}}\otimes P)\otimes p_2^*\Omega ^{-s} _{\mathbf{P}(V)}(-s),
\end{equation}
an object of $D(W_{X, n}\times \mathbf{P}(V)\times Y)$, and let $P_X^{s}\rightarrow P_X^{s+1}$ be the maps induced by the maps in \eqref{E:4.2.1}.

Projecting along   $W_{X, n}\times \mathbf{P}(V)\times Y\rightarrow W_{X, n}$ we find 
\begin{align*}
& \ \ \ \ \  \mathrm{RHom}_{W_{X, n}\times \mathbf{P}(V)\times Y}(P_X^i, P_X^j)\\
& \simeq R\Gamma (W_{X, n}\times \mathbf{P}(V)\times Y, p_1^*\Sigma _{X, n}^{\otimes (j-i)}\otimes p_2^*\mls RHom (\Omega ^{-i}_{\mathbf{P}(V)}(-i), \Omega ^{-j}_{\mathbf{P}(V)}(-j))\otimes p_{13}^*\mls RHom(P, P)) \\
& \simeq R\Gamma (W_{X, n}, \Sigma _{X, n}^{\otimes (j-i)}\otimes Rp_{1*}\mls RHom (P, P))\otimes _k\mathrm{RHom}(\Omega ^{-i}_{\mathbf{P}(V)}(-i), \Omega ^{-j}_{\mathbf{P}(V)}(-j)).
\end{align*}
Now  we have $Rp_{1*}\mls RHom (P, P)\in D^{\geq 0}(W_{X, n})$ and the natural map
$$
\mls O_{W_{X, n}}\rightarrow R^0p_{1*}\mls RHom (P, P)
$$
is an isomorphism (see for example the discussion in \cite[Remark 5.1]{lieblicholsson}).  We conclude that 
$$
\mathrm{Ext}^s(P^i_X, P^j_X) = 0
$$
for  $s<0$.  
Moreover, we have
\begin{align*}
\mathrm{Hom}(P^s_X, P^{s+1}_X)& \simeq \Gamma (W_{X, n}, \Sigma _{X, n})\otimes \mathrm{Hom}_{\mathbf{P}(V)}(\Omega ^{-s}(-s), \Omega ^{-(s+1)}(-(s+1)))\\
& \hookrightarrow \Gamma (X, K_X^{\otimes r})\otimes  \mathrm{Hom}_{\mathbf{P}(V)}(\Omega ^{-s}(-s), \Omega ^{-(s+1)}(-(s+1))).
\end{align*}

Using the map \eqref{E:4.30.15} we get a map
$$
\Gamma (X, K_X^{\otimes n})\otimes V^\vee \rightarrow \Gamma (X, K_X^{\otimes n})\otimes  \mathrm{Hom}_{\mathbf{P}(V)}(\Omega ^{-s}(-s), \Omega ^{-(s+1)}(-(s+1))).
$$
The image of the class 
 $\lambda _X\in \Gamma (X, K_X^{\otimes n})\otimes V^\vee $ adjoint to the inclusion $V\rightarrow \Gamma (X, K_X^{\otimes n})$ then equals the class of the differential $P^s_X\rightarrow P^{s+1}_X$.

By \cite[1.4]{BO}  the complex $P_X^\bullet $ in $D(W_{X, n}\times \mathbf{P}(V)\times Y)$ is induced by a unique object $\mls P_X\in DF(W_{X, n}\times \mathbf{P}(V)\times Y)$ of the filtered derived category.

Note that by \ref{C:3.17} this complex is supported on $W_{X, n}\times \mathbf{P}(V)\times W_{Y, n}$, and we view $\mls P_X$ more symmetrically as an object of $DF(W_{X, n}\times \mathbf{P}(V)\times W_{Y, n})$.
\end{pg}

\begin{pg}
We can also interchange $X$ and $Y$ and define
$$
P^s_Y:= p_{13}^*( P\otimes p_2^*\Sigma _{Y, n}^{\otimes s})\otimes p_2^*\Omega ^{-s}\Omega _{\mathbf{P}(V)}(-s).
$$
Using the isomorphism \eqref{E:3.2.1}
we view $P^s_Y$ as an object in $D(X\times \mathbf{P}(V)\times W_{Y, n})$.  As above we then get an object $\mls P_Y\in DF(W_{X, n}\times \mathbf{P}(V)\times W_{Y, n})$.

The isomorphism constructed in the proof of \cref{P:3.34}
$$
P\otimes p_1^*\Sigma _{X, n}\simeq P\otimes p_2^*\Sigma _{Y, n}
$$
induces isomorphisms
$$
\lambda _s:P^s_X\rightarrow P^s_Y.
$$
These isomorphisms are compatible with the transition maps (this follows from the construction of the isomorphism \eqref{E:3.2.1}) and therefore we get an isomorphism of complexes 
$$
\lambda _\bullet :P_X^\bullet \rightarrow P_Y^\bullet .
$$
By \cite[1.3]{BO} this is induced by a unique isomorphism
$$
\lambda :\mls P_X\rightarrow \mls P_Y
$$
in $DF(W_{X, n}\times \mathbf{P}(V)\times W_{Y, n})$.  If $B\subset \mathbf{P}(V)$ is an open set over which $f_B'$ and $g_{B}'$ are flat this induces an isomorphism
$$
\lambda :\delta _{0*}'P\rightarrow \delta _{1*}'P
$$
in $D(W_{X, n, B}\times B\times W_{Y, n, B})$.  
\end{pg}

\begin{lem}\label{L:4.4} If  $B\subset \mathbf{P}(V)$ is an open set over which $f'$ and $g'$ are flat, then  the following hold (let $\mathrm{pr}_3:W_{X, n, B}\times B\times W_{Y, n, B}\rightarrow W_{Y, n, B}$ denote the projection to the third factor):
\begin{enumerate}
    \item [(i)]   $R^i\mathrm{pr}_{3*}\mls RHom(\delta _{0*}'P, \delta _{1*}'P) = 0$ for $i<0$.
    \item [(ii)] The natural maps
$$
\xymatrix{
    R^0\mathrm{pr}_{3*}\mls RHom(\delta _{0*}'P, \delta _{1*}'P)\ar[r]& R^0\mls Hom(R\mathrm{pr}_{3*}\delta _{0*}'P,  R\mathrm{pr}_{3*}\delta _{1*}'P)& \mls O_{W_{Y, n, B}}\ar[l]}
    $$
are isomorphisms, where  the second map is obtained from the identification $\mathrm{pr}_3\circ \delta _1'\simeq \mathrm{pr}_3\circ \delta _0'$.
\item [(iii)] The isomorphism $\lambda $ satisfies the cocycle condition on $E(f_B, g_B)^2$.
\end{enumerate}
\end{lem}
\begin{proof}
Since $\delta _{0*}'P\simeq \delta _{1*}'P$, to prove (i) and (ii)  it suffices to prove the analogous statements with $\delta _{0*}'P$ replaced by $\delta _{1*}'P$.

Consider the diagram
$$
\xymatrix{
W_{X, n, B}\times W_{Y, n, B}\ar[d]^-{p_2}\ar@{^{(}->}[r]^-{\delta _1'}& W_{X, n, B}\times B\times W_{Y, n, B}\ar[d]^-{p_{23}}\ar@/^2pc/[dd]^-{p_3}\\
W_{Y, n, B}\ar@{^{(}->}[r]^-j& B\times W_{Y, n, B}\ar[d]^-{p_2}\\
& W_{Y, n, B},}
$$
where $j$ is the graph of $g_{B}'$.
We then have
\begin{align*}
    Rp_{23*}\mls RHom (\delta _{1*}'P, \delta _{1*}'P)&\simeq Rp_{23*}(\delta _{1*}'(\mls RHom (P, P)\otimes ^{\mathbf{L}}(\mathbf{L}\delta _1'^*\delta _{1*}\mls O_{W_{X, n, B}\times W_{Y, n, B}})^\vee )\\
    & \simeq j_*((Rp_{2*}\mls RHom (P, P))\otimes ^{\mathbf{L}}(\mathbf{L}j^*j_*\mls O_{W_{Y, n, B}})^\vee )
\end{align*}
and the natural map
$$
j_*\mls O_{W_{Y, n, B}}\rightarrow j_*((Rp_{2*}\mls RHom (P, P))\otimes ^{\mathbf{L}}(\mathbf{L}j^*j_*\mls O_{W_{Y, n, B}})^\vee )
$$
is an isomorphism in degrees $\leq 0$.   Pushing forward to $W_{Y, n, B}$ we get statements (i) and (ii).  Note that under these identifications  the element $1\in k$ corresponds to the previously constructed isomorphism $\lambda $.  

To complete the proof of it remains to show that the map $\lambda $ satisfies the cocycle condition on $E(f_B', g_B')^2$.  For this note that the preceding argument shows that the map
$$
\mathrm{Hom}_{D(E(f_B', g_B')^2)}(t_{2*}P, t_{0*}P)\rightarrow \mathrm{Hom}_{D(E(f_B', g_B')^0)}(P, P)
$$
inducing by pushing forward along the map $E(f_B', g_B')^2\rightarrow E(f_B', g_B')^0$ given by the unique map $[2]\rightarrow [0]$ is an isomorphims.  Since the pushforward of the map $\lambda $ is the identity on $P$ this implies that the cocycle condition holds.
\end{proof}

\begin{pg} From this and \ref{T:1.1} we conclude that there exists a dense open subset $B\subset \mathbf{P}(V)$ such that the restriction of $P$ to $W_{X, n, B}\times W_{Y, n, B}$ is induced by pushforward from $W_{X, n, B}\times _BW_{Y, n, B}$.  Let $A\subset \mathbf{P}$ be the preimage of $B$ and assume further that $B$ is chosen such that $\mathbf{P}\rightarrow \mathbf{P}(V)$ is \'etale over $B$.  Then (note that with our notation we have $W_{X, A} = W_{X, n, B}$)
$$
W_{X, A}\times _AW_{Y, A}\hookrightarrow W_{X, n, B}\times _BW_{Y, n, B}
$$
is open and closed and a complex on $W_{X, n, B}\times _BW_{Y, n, B}$ inducing $P$ is necessarily supported on $W_{X, A}\times _AW_{Y, A}$ (since we know that $P|_{W_{X, A}\times W_{Y, A}}$ is set-theoretically supported on $W_{X, A}\times _AW_{Y, A}$).  It follows that $P|_{W_{X, A}\times W_{Y, A}}$  is, in fact, the pushforward of a complex on $W_{X, A}\times _AW_{Y, A}$.

This completes the proof of \ref{T:3.4} (b) in the case of infinite $k$.
\end{pg}

\begin{pg}
To handle the case of finite $k$, note the following variant of \ref{L:4.4} above.  For an open subset $A\subset \mathbf{P}$ let $E(f_A, g_A)^\bullet $ be the cosimplicial scheme associated to the maps
$$
f_A:W_{X, A}\rightarrow A, \ \ g_A:W_{Y, A}\rightarrow A.
$$
\end{pg}

\begin{lem}\label{L:4.6} There exists a dense open subset $A\subset \mathbf{P}$ such that the following hold (let $\mathrm{pr}_3:W_{X,A}\times A\times W_{Y, A}\rightarrow W_{Y, A}$ denote the projection to the third factor)
\begin{enumerate}
    \item [(i)]   $R^i\mathrm{pr}_{3*}\mls RHom(\delta _{0*}P, \delta _{1*}P) = 0$ for $i<0$.
    \item [(ii)] The natural maps
$$
\xymatrix{
    R^0\mathrm{pr}_{3*}\mls RHom(\delta _{0*}P, \delta _{1*}P)\ar[r]& R^0\mls Hom(R\mathrm{pr}_{3*}\delta _{0*}P,  R\mathrm{pr}_{3*}\delta _{1*}P)& \mls O_{W_{Y, A}}\ar[l]}
    $$
are isomorphisms, where  the second map is obtained from the identification $\mathrm{pr}_3\circ \delta _1\simeq \mathrm{pr}_3\circ \delta _0$.
\item [(iii)] The map $\lambda :\delta _{0*}P\rightarrow \delta _{1*}P$, obtained from the isomorphisms in (ii) and the section $1\in \Gamma (W_{Y,A}, \mls O_{W_{Y, A}})$  is an isomorphism and satisfies the cocycle condition on $E(f_A, g_A)^2$.
\end{enumerate}
\end{lem}
\begin{proof}
It suffices to verify the lemma after passing to a field extension of $k$.  By the case of an infinite field we may therefore assume that there exists an open subset $A$ such that $P|_{W_{X, A}\times W_{Y, A}}$ is the pushforward of a complex on $W_{X, A}\times _AW_{Y, A}$.  In particular, we may assume that we have an isomorphism $\delta _{0*}P\simeq \delta _{1*}P$.  The proof now proceeds as in the proof of \ref{L:4.4}.
\end{proof}

Combining this with \ref{T:1.1} we then obtain \ref{T:3.4} (b) in the case of finite $k$ as well.  \qed

\section{Rouquier functors}\label{S:sec4}

In this section we explain how Rouquier's work \cite{rouquier} can be combined with our main result on support of complexes to obtain restrictions on kernels of derived equivalences.  This is also related to work of Lombardi \cite{Lombardi}.

\subsection{The Albanese torsor}

\begin{pg}
Let $k$ be a perfect field and let $X/k$ be a smooth projective variety.  Let $\mls Pic _X$ denote the $\mathbf{G}_m$-gerbe over the Picard scheme $\mathrm{Pic}(X)$ classifying line bundles on $X$, and set
$$
\mls Pic^0_X:= \mathrm{Pic}^0(X)\times _{\mathrm{Pic}(X)}\mls Pic _X.
$$
We assume that $\mathrm{Pic}^0(X)$ is a smooth scheme (this is automatic in characteristic $0$), and therefore an abelian variety, and write $\mathrm{Alb}(X)$ for the dual abelian scheme.
\end{pg}

\begin{pg}
For a smooth projective variety $X/k$ let $\mathbf{T}_X^0$ denote the functor 
which to any $k$-scheme $T$ associates the set of isomorphism classes of morphisms of Picard stacks
$$
s:\mathrm{Pic}_{X, T}^0\rightarrow \mls Pic _{X, T}^0
$$
over the identity.  Observe that any two such sections differ by a morphism of Picard stacks
$$
\rho :\mathrm{Pic}^0_{X, T}\rightarrow B\mathbf{G}_{m, T}.
$$
Considering the commutative diagram
$$
\xymatrix{
\mathrm{Pic}^0_{X, T}\times \mathrm{Pic}^0_{X, T}\ar[d]_{\rho \times \rho }\ar[r]^-m& \mathrm{Pic}^0_{X, T}\ar[d]^-\rho \\
B\mathbf{G}_{m, T}\times B\mathbf{G}_{m, T}\ar[r]^-{m_{B\mathbf{G}_m}}& B\mathbf{G}_{m, T}}
$$
and the fact that for the line bundle $\mls M$ on $B\mathbf{G}_{m, T}$ corresponding to the standard character of $\mathbf{G}_m$ we have 
$$
m_{B\mathbf{G}_{m}}^*(\mls M)\simeq \mls M\boxtimes \mls M
$$
it follows that $\rho $ corresponds to a line bundle $\mls L$ on $\mathrm{Pic}^0_{X, T}$ which is translation invariant; that is, a point of 
$$
\mathrm{Alb}_X:= \mathrm{Pic}^0_{\mathrm{Pic}^0_X}.
$$

Note also that a point $x\in X(k)$ yields a section $s$.  Indeed given $x$ we can interpret $\mathrm{Pic}^0_X$ as classifying pairs $(\mls L, \sigma )$ consisting of a line bundle $\mls L$ on $X$ and a trivialization $\sigma :\mls L(x)\simeq \kappa (x)$.  From this it follows that $\mathbf{T}_X^0$ is a torsor under $\mathrm{Alb}(X)$ and there is a natural morphism
$$
c_X:X\rightarrow \mathbf{T}_X^0.
$$
If we trivialize $\mathbf{T}^0_X$ using a point of $X$ then this is identified 
with the usual map from $X$ to its Albanese.

Note also that we have a canonical isomorphism (this amounts to the fact that 
the translation action of an abelian variety $A$ on $\mathrm{Pic}^0(A)$ is 
trivial)
$$
\mathrm{Pic}^0(\mathbf{T}_X^0)\simeq \mathrm{Pic}^0(\mathrm{Alb}_X)
$$
and therefore  an isomorphism 
$$
\mathrm{Pic}^0({\mathbf{T}_X^0})\simeq \mathrm{Pic}^0(X).
$$
Chasing through these identifications one finds that this is simply given by 
$$
c_X^*:\mathrm{Pic}^0_{\mathbf{T}^0_X}\rightarrow  \mathrm{Pic}_X^0.
$$
\end{pg}

\begin{pg}
We say that an 
autoequivalence
$$
\alpha :\D(X)\rightarrow \D(X)
$$
satisfies the \emph{Rouquier condition $R_X$} if the complex $Q_\alpha \in 
\D(X\times X)$ defining $\alpha $ is isomorphic to $\Gamma _{\sigma *}\mls L$, where $\Gamma 
_\sigma 
:X\rightarrow X\times X$ is the graph $x\mapsto (x, \sigma (x))$ of an 
automorphism $\sigma $ of $X$ and $\mls L$ is an invertible sheaf on $X$ 
numerically equivalent to $0$.

Let $\mls R_X^0$ be the fibered category which to any $k$-scheme $T$ associates 
the groupoid of objects $Q\in \D((X\times X)_T)$ of $T$-perfect complexes such 
that for all geometric points $\bar t\rightarrow T$ the fiber $Q_{\bar t}\in 
\D((X\times X)_{k(\bar t)})$ defines an equivalence $\D(X_{\bar t})\rightarrow 
\D(X_{\bar t})$ satisfying $R_X$ and whose associated automorphism $X_{\bar 
t}\rightarrow X_{\bar t}$ lies in the connected component of the identity in 
$\mathrm{Aut}(X)$.  Let $\mathbf{R}_X^0$ denote the group scheme
$$
\mathbf{R}_X^0:= \mathrm{Pic}^0(X)\times \mathrm{Aut}^0(X).
$$
Then $\mls R_X^0$ is a $\mathbf{G}_m$-gerbe over $\mathbf{R}^0_X$.  
\end{pg}

The key result of Rouquier that we will need is the following:
\begin{thm}[Rouquier] Let $Y/k$ be a second smooth projective variety related to $X$ by an equivalence $\Phi :D(X)\rightarrow D(Y)$.

(i) For any $T/k$ and $Q\in \D((X\times X)_T)$ in $\mls R_X^0(T)$ the complex 
$\Phi \circ Q\circ \Phi ^{-1}\in \D((Y\times Y)_T)$ is in $\mls R_Y^0(T)$.

(ii) The induced functor
\begin{equation}\label{E:gerbeequivalence}
\tilde \tau: \mls R_X^0\rightarrow \mls R_Y^0
\end{equation}
is an equivalence of gerbes.
\end{thm}
\begin{proof} See \cite[4.18]{rouquier}.
\end{proof}

 By passing to coarse moduli spaces the equivalence \eqref{E:gerbeequivalence} induces an isomorphism
$$
\tau :\mathbf{R}_X^0\rightarrow \mathbf{R}_Y^0.
$$

\begin{assume}
We assume for the rest of this section that $\mathrm{Pic}^0(X)$ is reduced and that $\tau $ takes $\mathrm{Pic}^0(X)$ to $\mathrm{Pic}^0(Y)$.
\end{assume}

\begin{remark} This assumption holds in many instances of interest.

(i) If $k$ has characteristic $0$ then the assumption that $\mathrm{Pic}^0(X)$ is reduced is automatic.

(ii) The assumption that $\mathrm{Pic}^0(X)$ is reduced implies that it is an abelian variety.  If this holds and furthermore $\mathrm{Aut}^0(Y)$ is affine, then automatically $\mathrm{Pic}^0(X)$ is mapped to $\mathrm{Pic}^0(Y)$.

(iii) In characteristic $0$ the condition that $\mathrm{Pic}^0(X)$ is taken to $\mathrm{Pic}^0(Y)$ can be checked on Hochschild cohomology.   
The map on tangent spaces at the identity of the morphism $\tau $  is a map
$$
T\tau :H^1(X, \mls O_X)\oplus H^0(X, T_X)\rightarrow H^1(Y, \mls O_Y)\oplus 
H^0(Y, T_Y).
$$
Using the HKR isomorphism this map is identified with the map on Hochschild cohomology
$$
HH^1(X)\simeq HH^1(Y).
$$
\end{remark}

\begin{pg}
Under this assumption the map $\tilde \tau $ induces an isomorphism of Picard stacks
$$
\tilde \gamma :\mls Pic _X^0\rightarrow \mls Pic _Y^0
$$
over an isomorphism of abelian varieties
$$
\gamma :\mathrm{Pic}_X^0\rightarrow \mathrm{Pic}^0_Y.
$$
It therefore also induces an isomorphism of torsors of sections
$$
\rho :\mathbf{T}_X^0\rightarrow \mathbf{T}_Y^0
$$
compatible with the isomorphism
$$
\gamma ^t:\mathrm{Alb}_X\rightarrow \mathrm{Alb}_Y.
$$
\end{pg}

The main result of this section is the following:
\begin{thm}\label{T:4.9}
There exists a dense open subset $A\subset \mathbf{T}_Y^0$ such that the restriction of $P$ to $c_{X}^{-1}\rho ^{-1}(A)\times c_Y^{-1}(A)\subset X\times Y$ is in the image of 
$$
D(c_{X}^{-1}\rho ^{-1}(A)\times _Ac_Y^{-1}(A))\rightarrow D(c_{X}^{-1}\rho ^{-1}(A)\times c_Y^{-1}(A)).
$$
\end{thm}

The proof occupies the remainder of the section.

\begin{lem}
Let $\mls L_X^u$ (resp. $\mls L_Y^u$) be the universal line bundle on $X\times \mls Pic^0_X$ (resp. $\mls Pic ^0_Y\times Y$).   Then we have a canonical isomorphism
\begin{equation}\label{E:tautiso}
p_{12}^*\mls L_X^u\lotimes p_{13}^*P\simeq (1_X\times \tilde \gamma \times 1_Y)^*(p_{23}^*\mls L_Y^u\lotimes p_{13}^*P)
\end{equation}
in $D(X\times \mls Pic _X^0\times Y)$.
\end{lem}
\begin{proof}
To ease notation let us write $\mls P_X$ (resp. $\mls P_Y$) for $\mls Pic _X^0$ (resp. $\mls Pic _Y^0$).  The isomorphism $\tilde \gamma $ is characterized by the condition that the complex
\begin{equation}\label{E:5.18.1}
\Phi \circ (\otimes \mls L_X^u)\circ \Phi ^{-1} \in D((Y\times Y)_{\mls P_Y})
\end{equation}
is isomorphic to $\Delta _{Y*}\mls L_Y^u$. 
Consider the cartesian square
$$
\xymatrix{
(X\times Y)_{\mls P_Y}\ar[d]^-{p_2}\ar[r]^-{(x, y)\mapsto (y, x, y)}& (Y\times X\times Y)_{\mls P_Y}\ar[d]^-{p_{13}}\\
Y_{\mls P_Y}\ar[r]^-{\Delta _Y}& (Y\times Y)_{\mls P_Y}.}
$$
Then \eqref{E:5.18.1} is represented by the complex
$$
Rp_{13*}(p_2^*\mls L_X^u\lotimes p_{12}^*P\lotimes p_{23}^*P^\vee \otimes p_3^*\omega _Y[\mathrm{dim}(Y)])\in D((Y\times Y)_{\mls P_Y}).
$$
By Grothendieck duality and the isomorphism 
$$
p_{13}^!(-)\simeq p_{13}^*(-)\lotimes \omega _X[\mathrm{dim}(X)]
$$
we find that the characterizing isomorphism 
$$
Rp_{13*}(p_2^*\mls L_X^u\lotimes p_{12}^*P\lotimes p_{23}^*P^\vee \otimes p_3^*\omega _Y[\mathrm{dim}(Y)])\simeq \Delta _{Y*}\mls L_Y^u
$$
corresponds by adjunction to a morphism
$$
(p_2^*\mls L_X^u\lotimes p_{12}^*P\lotimes p_{23}^*P^\vee \otimes p_3^*\omega _Y[\mathrm{dim}(Y)])\rightarrow p_{13}^*(\Delta _{Y*}(1_Y\times \tilde \gamma )^*\mls L_Y^u)\lotimes p_2^*\omega _X[\mathrm{dim}(X)]
$$
in $D((Y\times X\times Y)_{\mls P_X})$.  
Using the isomorphism $P\otimes \omega _Y[\mathrm{dim}(Y)]\simeq P\otimes \omega _X[\mathrm{dim}(X)]$ and adjunction this, in turn, corresponds to  a map 
$$
p_{12}^*\mls L_X^u\lotimes p_{13}^*P\rightarrow  (1_X\times \tilde \gamma \times 1_Y)^*(p_{23}^*\mls L_Y^u\lotimes p_{13}^*P)
$$
in $D(X\times \mls Pic _X^0\times Y)$.
This map is an isomorphism, since it can be verified in each of the fibers where it holds by Orlov's theorem and the fact that they both determine the same functor.
\end{proof}
\begin{lem}\label{L:tensorfromT} Let $S$ be a  noetherian scheme and let $\mls F\in D(\mathbf{T}_{X, S}^0)$ be a complex with associated complex $\mls F^\rho \in D(\mathbf{T}_{Y, S}^0)$.  Then we have
\begin{equation}\label{E:3.1.1}
p_{12}^*c_{X, S}^*\mls F\lotimes p_{13}^*P\simeq p_{23}^*c_{Y, S}^*\mls F^\rho \lotimes p_{13}^*P
\end{equation}
in $D(X\times S\times Y)$.
\end{lem}
\begin{proof}
Note that the diagram
$$
\xymatrix{
\mls Pic ^0_{\mathbf{T}_X^0}\ar[d]\ar[r]^-{c_X^*}& \mls Pic ^0_{X}\ar[d]\\
\mathrm{Pic}_{\mathbf{T}_X^0}^0\ar[r]^-{c_X^*}& \mathrm{Pic}^0_X}
$$
is cartesian, and identifies $\mathbf{T}_X^0$ with the $\mathbf{G}_m$-torsor of sections of 
$$
\mls Pic ^0_{\mathbf{T}_X^0}\rightarrow \mathrm{Pic}_{\mathbf{T}_X^0}^0.
$$
In particular, there is a universal line bundle $\mls L_X^u$ on $\mathrm{Pic}_{\mathbf{T}_X^0}^0\times \mathbf{T}_X^0.$  Similarly there is a universal line bundle $\mls L_Y^u$ over $\mathrm{Pic}_{\mathbf{T}_Y^0}^0\times \mathbf{T}_Y^0$, and the isomorphism
$$
(\rho ^*\times \rho ):\mathrm{Pic}_{\mathbf{T}_X^0}^0\times \mathbf{T}_X^0\rightarrow \mathrm{Pic}_{\mathbf{T}_Y^0}^0\times \mathbf{T}_Y^0
$$
comes equipped with an isomorphism
$$
(\rho ^*\times \rho )^*\mls L_Y^u\simeq \mls L_X^u.
$$
The functor
$$
D(\mathrm{Pic}_{X, S}^0)\rightarrow D(\mathbf{T}_{X, S}^0), \ \ \mls G\mapsto Rp_{2*}(p_1^*\mls G\otimes p_2^*\mls L_X^u)
$$
is an equivalence of categories.  Indeed this can be verified after making a field extension, where it reduces to the standard derived equivalence between an abelian variety and its dual.  In particular, we can write
$$
\mls F = Rp_{2*}(p_1^*\mls G\otimes p_2^*\mls L_X^u)
$$
for a unique object $\mls G\in D(\mathrm{Pic}^0_{X, S})$.  Note also that if $\mls G^\rho \in D(\mathrm{Pic}^0_{Y, S}))$ is the complex corresponding to $\mls G$ under the isomorphism 
$$
\rho ^*:\mathrm{Pic}_{X}^0\rightarrow \mathrm{Pic}_Y^0
$$
induced by $\rho $, then $\mls G^\rho $ transforms to $\mls F^\rho $ on $\mathbf{T}_Y^0$ under the equivalence defined by $\mls L_Y^u$.

Consider the diagram
$$
\xymatrix{
X\times \mathrm{Pic}^0_{\mathbf{T}_X^0}\times S\times Y\ar[d]\ar[r]& X\times S\times Y\ar[r]\ar[d]& X\times Y\\
X\times \mathrm{Pic}^0_{\mathbf{T}_X^0}\times S\ar[d]\ar[r]& X\times S\ar[d]& \\
\mathbf{T}^0_X\times \mathrm{Pic}^0_{\mathbf{T}_X^0}\times S\ar[d]\ar[r]& \mathbf{T}_X^0\times S& \\
\mathrm{Pic}^0(\mathbf{T}_X^0)\times S.
}
$$
From this we see that the complex on the left side of \eqref{E:3.1.1} is isomorphic to the complex
$$
Rp_{134*}(p_{14}^*P\otimes p_{12}^*\mls L_X^u\otimes p_{23}^*\mls G).
$$
Using the isomorphism \eqref{E:tautiso} we find that the image of
$$
p_{14}^*P\otimes p_{12}^*\mls L_X^u\otimes p_{23}^*\mls G
$$
in 
$$
D(X\times \mathrm{Pic}^0(\mathbf{T}_Y^0)\times S\times Y)
$$
is equal to
$$
p_{14}^*P\otimes p_{24}^*\mls L_Y^u\otimes p_{23}^*\mls G.
$$
From this the result follows.
\end{proof}

\begin{lem}\label{L:support} Let $x\in X$ be a point with image $z\in \mathbf{T}_X^0$.  Then the complex $P_{x}\in Y_{\kappa (x)}$ is set-theoretically supported on  $c_{Y}^{-1}(\rho (z)).$
\end{lem}
\begin{proof}
After making the field extension from $k$ to $\kappa (x)$, we may assume that $x$ is a $k$-rational point.

The support of $P_x$ in $X\times Y$ is contained in the support of 
$$
P_{c_X^{-1}(z)} = P|_{c_X^{-1}(z)\times Y},
$$
so it suffices to show that the support of $P_{c_X^{-1}(z)}$ is contained in $X\times c_Y^{-1}(\rho (z))$.

For this apply \cref{L:tensorfromT} with $S =\Sp (k)$ and $\mls F$ the skyscraper sheaf $\kappa (z)$ on $\mathbf{T}_X^0$.  We then find that the support of  $P_{c_X^{-1}(z)}$  is equal to the support of 
$$
P|_{X\times c_Y^{-1}(\rho (z))}.
$$
\end{proof}

\begin{pg} Let $W_X\subset \mathbf{T}_X^0$ (resp. $W_Y\subset \mathbf{T}_Y^0$) be the scheme-theoretic image of $c_X$ (resp. $c_Y$).  If $x\in X$ is a point then it follows from \cref{L:support} that $c_Y^{-1}(\rho (c_X(x)))$ is nonempty; that is, $\rho (c_X(x))\in W_Y$.  Since $W_X$ and $W_Y$ are integral it follows that $\rho $ restricts to a morphism
\begin{equation}\label{E:Wiso}
W_X\rightarrow W_Y,
\end{equation}
which we again denote by $\rho $.  By considering the inverse transform we see that this map is an isomorphism.  
\end{pg}

\begin{pg}
Let $f:X\rightarrow W_X$ be the map induced by $c_X$, and let 
$$
g:Y\rightarrow W_X
$$
denote the composition of $c_Y:Y\rightarrow W_Y$ with the inverse of \eqref{E:Wiso}, and let $E^\bullet $ be the associated cosimplicial scheme as in \ref{P:setup}.  Applying \ref{L:tensorfromT} with $S = W_X$ and $\mls F$ the sheaf $u _*\mls O_{W_X}$, where $u :W_X\rightarrow \mathbf{T}_X^0\times W_X$ is the graph of the inclusion, we find that on 
$$
E^1 = X\times W_X\times Y
$$
we have
$$
\delta _{0*}P\simeq \delta _{1*}P.
$$
\end{pg}
\begin{pg}
Having established the existence of this isomorphism we can proceed as in the case of the canonical fibration.  Namely if $A\subset W_X$ is an open subset over which $f$ and $g$ are flat, and 
$$
f_A:f^{-1}(A)\rightarrow A, \ \ g_A:g^{-1}(A)\rightarrow A
$$
are the restrictions, then the same argument shows that the map
$$
\mathrm{Hom}_{E(f_A, g_A)^2}(t_{2*}P, t_{0*}P)\rightarrow \mathrm{Hom}_{f^{-1}(A)\times g^{-1}(A)}(P, P),
$$
induced by the surjection $[2]\rightarrow [0]$,  is an isomorphism.   From this it follows that the isomorphism $\delta _{0*}P\simeq \delta _{1*}P$ satisfies the cocycle condition, after restriction to $A$. Theorem \ref{T:4.9} then follows using \ref{T:1.1}. \qed
\end{pg}

\bibliographystyle{amsplain}
\bibliography{bibliography}{}

\providecommand{\bysame}{\leavevmode\hbox to3em{\hrulefill}\thinspace}
\providecommand{\MR}{\relax\ifhmode\unskip\space\fi MR }
\providecommand{\MRhref}[2]{%
  \href{http://www.ams.org/mathscinet-getitem?mr=#1}{#2}
}
\providecommand{\href}[2]{#2}
\begin{thebibliography}{10}

\bibitem{BZFN}
D.~Ben-Zvi, J.~Francis, and D.~Nadler, \emph{Integral transforms and {D}rinfeld
  centers in derived algebraic geometry}, J. Amer. Math. Soc. \textbf{23}
  (2010), no.~4, 909--966. \MR{2669705}

\bibitem{Beilinsonresolution}
A.~A. Be\u{\i}linson, \emph{Coherent sheaves on {${\bf P}^{n}$} and problems in
  linear algebra}, Funktsional. Anal. i Prilozhen. \textbf{12} (1978), no.~3,
  68--69. \MR{509388}

\bibitem{BCHM}
C.~Birkar, P.~Cascini, C.~Hacon, and J.~McKernan, \emph{Existence of minimal
  models for varieties of log general type}, J. Amer. Math. Soc. \textbf{23}
  (2010), no.~2, 405--468. \MR{2601039}

\bibitem{bondalorlov01}
A.~Bondal and D.~Orlov, \emph{Reconstruction of a variety from the derived
  category and groups of autoequivalences}, Compositio Math. \textbf{125}
  (2001), no.~3, 327--344. \MR{1818984}

\bibitem{alterations}
A.~J. de~Jong, \emph{Smoothness, semi-stability and alterations}, Inst. Hautes
  \'{E}tudes Sci. Publ. Math. (1996), no.~83, 51--93. \MR{1423020}

\bibitem{huybrechtsFM}
D.~Huybrechts, \emph{Fourier-{M}ukai transforms in algebraic geometry}, Oxford
  Mathematical Monographs, The Clarendon Press, Oxford University Press,
  Oxford, 2006. \MR{2244106}

\bibitem{Illusie}
L.~Illusie, \emph{Complexe cotangent et d\'{e}formations. {I}}, Lecture Notes
  in Mathematics, Vol. 239, Springer-Verlag, Berlin-New York, 1971.
  \MR{0491680}

\bibitem{lieblicholsson}
M.~Lieblich and M.~Olsson, \emph{Fourier-{M}ukai partners of {K}3 surfaces in
  positive characteristic}, Ann. Sci. \'{E}c. Norm. Sup\'{e}r. (4) \textbf{48}
  (2015), no.~5, 1001--1033. \MR{3429474}

\bibitem{Lombardi}
L.~Lombardi, \emph{Derived invariants of irregular varieties and {H}ochschild
  homology}, Algebra Number Theory \textbf{8} (2014), no.~3, 513--542.
  \MR{3218801}

\bibitem{BO}
M.~Olsson, \emph{The {BBD} gluing lemma}, preprint, 2021.

\bibitem{orlov}
D.~O. Orlov, \emph{Derived categories of coherent sheaves and equivalences
  between them}, Uspekhi Mat. Nauk \textbf{58} (2003), no.~3(351), 89--172.
  \MR{1998775}

\bibitem{popaschnell}
M.~Popa and C.~Schnell, \emph{Derived invariance of the number of holomorphic
  1-forms and vector fields}, Ann. Sci. \'Ec. Norm. Sup\'er. (4) \textbf{44}
  (2011), no.~3, 527--536. \MR{2839458}

\bibitem{rizzardo}
A.~Rizzardo and M.~Van~den Bergh, \emph{Scalar extensions of derived categories
  and non-{F}ourier-{M}ukai functors}, Adv. Math. \textbf{281} (2015),
  1100--1144. \MR{3366860}

\bibitem{rouquier}
R.~Rouquier, \emph{Automorphismes, graduations et cat\'egories triangul\'ees},
  J. Inst. Math. Jussieu \textbf{10} (2011), no.~3, 713--751. \MR{2806466
  (2012h:16019)}

\bibitem{stacks-project}
The {Stacks Project Authors}, \emph{\itshape stacks project},
  \url{http://stacks.math.columbia.edu}, 2016.

\bibitem{Toda}
Y.~Toda, \emph{Fourier-{M}ukai transforms and canonical divisors}, Compos.
  Math. \textbf{142} (2006), no.~4, 962--982. \MR{2249537 (2007d:14039)}

\bibitem{toen}
B.~To{\"e}n, \emph{The homotopy theory of {$dg$}-categories and derived
  {M}orita theory}, Invent. Math. \textbf{167} (2007), no.~3, 615--667.
  \MR{2276263 (2008a:18006)}

\end{thebibliography}

\end{document}